\documentclass[11pt,a4paper, reqno]{amsart}
\usepackage{mathrsfs}

\usepackage{amsmath,amsfonts,verbatim}
\usepackage{latexsym}
\usepackage{amssymb,leftidx}
\usepackage{extarrows}
\usepackage{overpic}
\usepackage{color}
\usepackage{epsfig}
\usepackage{subfigure}
\usepackage{tikz}
\usepackage{bbm}
\usepackage{tikz, caption}
\usepackage[font=small,labelfont=bf]{caption}

\usepackage[colorlinks=true,backref=page]{hyperref}
\hypersetup{urlcolor=blue, citecolor=red}

\usepackage{amssymb}
\usepackage{mathrsfs}
\usepackage{amscd}
\usepackage{cite}

\newcommand\R{\mathbb{R}}
\newcommand\C{\mathbb{C}}

\newcommand\N{\mathbb{N}}

\newcommand\A{{\bf A}}

\newcommand\LL{\mathcal{L}}

\usepackage{graphicx}
\usepackage{float}

\numberwithin{equation}{section}
\newtheorem{proposition}{Proposition}[section]

\newtheorem{lemma}{Lemma}[section]
\newtheorem{theorem}{Theorem}[section]

\newtheorem{remark}{Remark}[section]

\begin{document}
\title[Resolvent estimates for 2D-magnetic Schr\"odinger ]{Uniform resolvent estimates for critical magnetic Schr\"odinger operators in 2D}

\author{Luca Fanelli}
\address{Ikerbasque and Universidad del Pa\'is Vasco, Departamento de Matem\'aticas, Bilbao, Spain}
\email{luca.fanelli@ehu.es}

\author{Junyong Zhang}
\address{Department of Mathematics, Beijing Institute of Technology, Beijing 100081}
\email{zhang\_junyong@bit.edu.cn}

\author{Jiqiang Zheng}
\address{Institute of Applied Physics and Computational Mathematics, Beijing 100088}
\email{zhengjiqiang@gmail.com; zheng\_jiqiang@iapcm.ac.cn}

\begin{abstract}
We study the $L^p-L^q$-type uniform resolvent estimates for 2D-Schr\"odinger operators in scaling-critical magnetic fields, involving the Aharonov-Bohm model as a main example. As an application, we prove localization estimates for the eigenvalue of some non self-adjoint zero-order perturbations of the magnetic Hamiltonian.

\end{abstract}


\maketitle 

\section{Introduction}

In dimension $n=2$, let us consider a scaling invariant magnetic Schr\"odinger operator of the form
\begin{equation}\label{LAa}
\mathcal{L}_{{\A},a}=\Big(i\nabla+\frac{{\A}(\hat{x})}{|x|}\Big)^2+\frac{a(\hat{x})}{|x|^2},
\end{equation}
where $\hat{x}=\tfrac{x}{|x|}\in\mathbb{S}^1$, $a\in W^{1,\infty}(\mathbb{S}^{1},\mathbb{R})$ and ${\A}\in W^{1,\infty}(\mathbb{S}^1;\R^2)$  satisfies the transversality condition
\begin{equation}\label{eq:transversal}
{\A}(\hat{x})\cdot\hat{x}=0,
\qquad
\text{for all }x\in\R^2.
\end{equation}
We recently started a program involving the above operators.
 In  \cite{FZZ},  the authors proved
the Strichartz estimates for wave; and later  in \cite{GYZZ}, together with Gao and Yin, the last two authors showed the Strichartz estimates for Klein-Gordon by constructing the spectral measure of the operator $\LL_{{\A},0}$. In \cite{GWZZ},
 we studied $L^2$-weighted resolvent estimates.
In this paper, based on the resolvent kernel constructed in \cite{GYZZ}, we aim to study  the $L^p\to L^q$ type uniform resolvent inequalities for the purely magnetic operator $\mathcal L_{\A,0}$ of the form
\begin{equation}\label{eq:unifmagn}
\|(\LL_{{\A},0}-\sigma)^{-1}\|_{L^q(\R^2)\to L^p(\R^2)}\leq C|\sigma|^{\frac1p-\frac1q-1},
\end{equation}
where $\sigma\in\C\setminus[0,\infty)$ and the constant $C$ is independent of $\sigma$.
\vspace{0.1cm}

The zero-energy and general dimension version of estimates \eqref{eq:unifmagn}  are related to the classical Hardy-Littlewood-Sobolev inequality, which by Fourier transform is equivalent to the Sobolev embedding
\begin{equation}
\|(-\Delta)^{-\frac s2} u\|_{L^q(\R^n)}\lesssim \|u\|_{L^p(\R^n)},\quad \frac1p=\frac1q+\frac sn,\quad 1<q,r<\infty,\quad 0<s<n,
\end{equation}
or equivalently
$$\|u\|_{L^q(\R^n)}\lesssim \|(-\Delta)^{\frac s2} u\|_{L^p(\R^n)}.$$
Here the symbol $\lesssim$ means as usual that in the right-hand side of the inequality a multiplicative constant is hidden, which does not depend on $u$.
In particular, when $s=2$, this inequality can be regarded as the following resolvent estimate
$$\|(-\Delta-\sigma)^{-1} u\|_{L^q(\R^n)}\leq  C(\sigma)\|u\|_{L^p(\R^n)}$$
at the zero energy $\sigma=0$, where $C(\sigma)$ is positive dimensional constant depending on $\sigma$. For non-zero energies, we have the following bound
\begin{equation}\label{uS}
\left\|(-\Delta-\sigma)^{-1}u\right\|_{L^{q} (\R^n)}\leq
C|\sigma|^{\frac n2(\frac1p-\frac1q)-1}\|u\|_{L^{p}(\R^n)},\quad \sigma\in\C\setminus[0,\infty),
\end{equation}
where $n\ge3$ and $(p,q)$ satisfies
\begin{align*}\label{p_q_0}
\frac{2}{n+1}\le \frac1p-\frac1q\le \frac2n,\ \frac{2n}{n+3}<p<\frac{2n}{n+1},\ \frac{2n}{n-1}<q<\frac{2n}{n-3}.
\end{align*}

In this free context ${\A}\equiv0$, the above result appears first in the seminal paper by Kenig-Ruiz-Sogge \cite{KRS} (in the case $q=p'$, $n\geq3$), and later in Gutierrez for the complete range of $p,q$, and again $n\geq3$ (we also mention the paper \cite{CS, Gut}, as well as the well known paper by Kato-Yajima \cite{KY} and the references therein for a more general picture). Later, Frank noticed in \cite{Fr} that the same holds in dimension $n=2$, away from the endpoint, and used this to prove eigenvalue bounds for zero-order perturbations.
Again in dimension $n=2$, more recently Ev\'equoz \cite{Ev} applied the method in \cite{Gut}
to show the uniform resolvent estimate \eqref{uS}
  provided that $(1/p,1/q)$ is contained in the pentagon (see Figure 1)
\begin{equation}\label{pq2}
	\{(x,y): 2/3\le x-y <1, \, 3/4<x\le1, \, 0\le y<1/4 \}.	
\end{equation}	

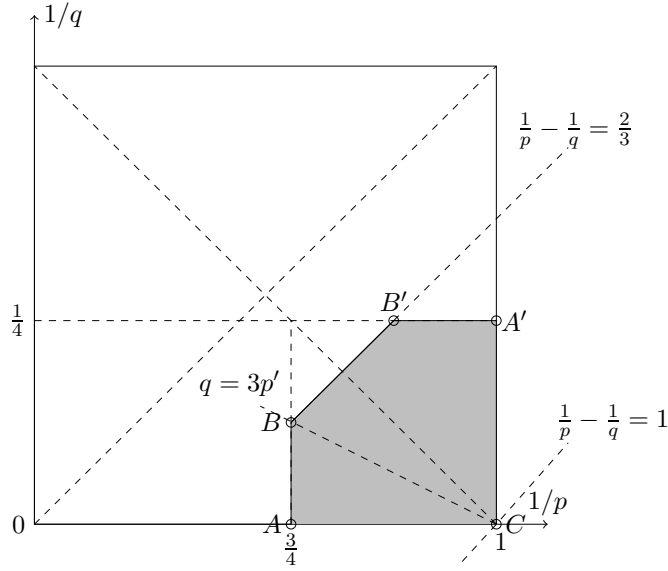
\begin{figure}[htbp]\label{figure1}
\begin{center}
\scalebox{0.9}[0.9]{
\begin{tikzpicture}
\draw (0,0) rectangle (6.75,6.75);
\draw[->]  (0,0) -- (0,7.5);
\draw[->]  (0,0) -- (7.5,0);
\filldraw[fill=gray!50](3.75,0)--(3.75,1.5)--(5.25,3)--(6.75,3)--(6.75,0); 
\draw (6.75,0) node[below] {$\ 1$};
\draw (0,0) node[below, left] {$0$};
\draw (7.5,0) node[above] {$1/p$};
\draw (0,7.5) node[right] {$1/q$};
\draw[dashed] (0,0) -- (6.75,6.75);
\draw[dashed] (6.75,0) -- (0,6.75);
\draw[dashed] (6.25,-0.55)--(6.75,0)  -- (7.8,1.2);
\draw (7.5,1.5) node[right] {$\frac1p-\frac1q=1$};
\draw[dashed] (5.25,3) -- (7.8,5.55);
\draw (7.9,5.4) node[above] {$\frac1p-\frac1q=\frac{2}{3}$};
\draw (3.0,1.76) node[above] {$q=3p'$};
\draw[dashed] (3.75,1.5) -- (3.75,3.0) -- (5.25,3.0);
\draw (3.75,0) circle (2pt) node[left] {$A$};    
\draw (6.75,3)  circle (2pt) node[right] {$\!A'$}; 
\draw (3.75,1.5)  circle (2pt) node[left] {$B$};   
\draw (5.25,3)  circle (2pt) node[above] {$B'$}; 
\draw (6.75,0)  circle (2pt) node[right] {$C$};
\draw[dashed] (6.75,0) -- (3.3,1.74);
\draw[dashed] (3.75,0) -- (3.75,1.5) ; 
\draw[dashed] (5.25,3) -- (6.75,3) ; 
\draw (3.75,1.5) -- (5.25,3); 
\draw (3.75,0) node[below] {$\frac{3}{4}$};
\draw (0,3) node[left] {$\frac{1}{4}$};
\draw[dashed] (0,3) -- (3.75,3.0);
\end{tikzpicture}
}
\end{center}
\caption{Here $A=(\frac{3}{4},0)$, $B=(\frac{3}{4},\frac{1}{12})$, and $A',B'$ are dual points of $A,B$, respectively.}
\end{figure}

By the scaling $u(z)\mapsto u(|\sigma|^{1/2}z)$, we see that to prove \eqref{uS} it is sufficient to prove
the case $|\sigma|=1, \sigma\neq 1$. In the case $\frac1p-\frac1q=\frac2n$, the constant of the estimate is uniform in $\sigma$, and \eqref{uS} is equivalent to
\begin{equation}\label{uS'}
\|u\|_{L^q(\R^n)}\leq  C\|(-\Delta-\sigma) u\|_{L^p(\R^n)},  \quad |\sigma|=1, \sigma\neq 1.
\end{equation}
This is why, when  $\frac1p-\frac1q=\frac2n$, \eqref{uS} is usually called the \emph{uniform Sobolev inequality}.
Notice that in dimension $n=2$, which is the one studied in this manuscript, the uniform Sobolev inequality (as well as the Sobolev inequality) fails (it is the low-right corner $\frac1p-\frac1q=1$ in Figure 1), due to the logaritmic singularity in the kernel of the Green function.
We also stress that  \eqref{pq2} is known to be sharp, see \cite{Gut,KRS,KL}.

It is a quite natural problem to ask whether the analogous uniform resolvent estimates hold for the Schr\"odinger operator with magnetic or electric potentials.
 There is an extensive  literature (see e.g.  \cite{RV, Kato} and the references therein for an obviously incomplete picture) on the resolvent estimates of the Schr\"odinger operator with purely electric potential
\begin{equation}
\LL_V=-\Delta+V(x).
\end{equation}
It was observed that the uniform resolvent estimates play a fundamental role in the establishment of time-decay estimates or Strichartz estimates, we refer to \cite{Kato, JSS, RS, GS, JK}.
For electromagnetic or magnetic cases, there are some results proved by Garcia \cite{Gar1, Gar2} and Cuenin-Kenig \cite{CK16}.
 In all the above papers, the properties of the potentials about the smooth regularity and decay conditions at infinity are key to this problem.
On one hand, the potentials can usually be quite rough, but have sufficient decay at infinity. On the other hand, unbounded potentials are allowed but they are usually assumed to be smooth. The general philosophy to prove resolvent estimates in $L^p$ spaces is to combine the $L^p$ estimates for the free resolvent with weighted $L^2$ estimates.\vspace{0.2cm}

In this paper, we focus on purely magnetic potentials that are at critical with respect to the scaling of the free operator. It worths to remark that
most of results \cite{Gar1,Gar2, CK16} needed that the electric potential decays at infinity as $|x|^{-2-\epsilon}$ and the magnetic potential
$|x|^{-1-\epsilon}$ with $\epsilon>0$.
The inverse-square potential $V(x)=a(\hat{x})|x|^{-2}$ with $\hat{x}=\frac{x}{|x|}$ is a typical scaling-critical electric potential; and any homogeneous vector field of degree -1 is a scaling critical magnetic potential.
In \cite{BPST},  Burq-Planchon-Stalker-Tahvildar-Zadeh  proved
the $L^2$-weighted resolvent estimates $$\||x|^{-1}((-\Delta+a(\hat{x})|x|^{-2})-\sigma)^{-1}|x|^{-1}\|_{L^2(\R^n)\to L^2(\R^n)}\leq C$$ hold uniformly in $\sigma$, and then proved local smoothing estimates and Strichartz estimates.
Recently,  Bouclet-Mizutani \cite{BM1} proved
the weighted resolvent estimates $\|w(\LL_V-\sigma)^{-1}w^*\|_{L^2\to L^2}$ uniformly in $\sigma$
where $w$ is in a large class of weight functions in Morrey-Campanato spaces.
The assumption on potential $V(x)$ in \cite{BM1}  covers the scaling critical inverse-square type potentials $a|x|^{-2}$.
Later,  Mizutani \cite{Miz} generalized the  $L^2$ weighted resolvent estimates to  the $L^p-L^q$ type uniform Sobolev estimates \eqref{uS},
but one needs  that $V\in L^{\frac n2}(\R^n)$ with $n\geq3$ which does not include the inverse-square potential $a|x|^{-2}$.
Very recently, Mizutani and the last two authors \cite{MZZ}
proved the uniform resolvent estimates \eqref{uS} which recovers the scaling critical inverse-square type potentials
and the dimension $n\geq3$. The new ingredient of \cite{MZZ} was to prove a generalized $L^2$ weighted resolvent estimates
$$\||x|^{-s}((-\Delta+a(\hat{x})|x|^{-2})-\sigma)^{-1}|x|^{-(2-s)}\|_{L^2(\R^n)\to L^2(\R^n)}\leq C,$$
hold for some $s$, hence one can obtain  $L^p-L^q$ type uniform resolvent estimates. The general philosophy, mentioned above, combining the resolvent estimates of $-\Delta$ in $L^p$ spaces
and weighted $L^2$ estimates works for $n\geq 3$ in \cite{MZZ} but fails when $n=2$.

 \vspace{0.2cm}

 In \cite{GWZZ}, Gao, Wang and the last two author considered the electromagnetic operator $\LL_{{\A},a}$ in $\R^2$ given by \eqref{LAa} and prove the following weighted $L^2$-estimates
\begin{equation}\label{w-L2}
\||x|^{-s}(\LL_{{\A},a}-\sigma)^{-1}|x|^{-(2-s)}\|_{L^2(\R^2)\to L^2(\R^2)}\leq C,
\end{equation}
where $s\in (1-\mu_0,1+\mu_0)$ with $\mu_0$ depending on ${\A}$ and $a$.
 In particular, the diamagnetic phenomenon allows to consider negative electric potential which can be singular in the same fashion as the inverse-square potential. It is worth pointing out that the model $\LL_{{\A},a}$ is in dimension two in which \eqref{w-L2} fails even for $s=1$ when
 ${\bf A}=0$ and $a=0$ since the resonance occurs at zero energy  in this case.

\vspace{0.2cm}

We focus on the purely magnetic case and
 we  study the validity of the uniform resolvent estimates \eqref{uS} for
 the magnetic Schr\"odinger operator $\LL_{{\A},0}$ in $\R^2$. The main example is given by the Aharonov-Bohm potential
\begin{equation}\label{ab-potential}
a\equiv0,
\qquad
{\A}(\hat{x})=\alpha\Big(-\frac{x_2}{|x|},\frac{x_1}{|x|}\Big),\quad \alpha\in\R,
\end{equation}
introduced in \cite{AB}, as an well known model describing some scattering phenomenon arising in regions where the magnetic field is not present.
Notice that the singularity of the Aharonov-Bohm potential is not locally in $L^2$, which makes the analysis of this operator peculiar, compared to the usual theory for magnetic Schr\"odinger Hamiltonians (see \cite{RS}). On the other hand, the scaling invariance helps from the point of view of Fourier Analysis and permits to perform explicit calculus for generalized eigenfunctions, as we shall see in the sequel.

 \vspace{0.2cm}

Our first main result is the following.
\begin{theorem}\label{thm:LA0} Let $1\leq p<\frac43$ and $4<q\leq \infty$ satisfy $\frac23\leq \frac1p-\frac1q<1$ and let $\LL_{{\A},0}$
be \eqref{LAa} with $a\equiv0$.
Then, for $\sigma\in\C\setminus [0,+\infty)$, there exists a constant $C=C(p,q)>0$ such that
\begin{equation}\label{est:resa-z}
\|(\LL_{{\A},0}-\sigma)^{-1}f\|_{L^q(\R^2)}\leq C|\sigma|^{\frac1p-\frac1q-1}\|f\|_{L^p(\R^2)}.
\end{equation}
In particular, if $6\leq q<\infty$, then \eqref{est:resa-z} holds for $q=p'$.
\end{theorem}

\begin{remark}
It would be interesting to prove \eqref{est:resa-z} for $\LL_{{\A},a}$. However,
the resolvent representation was constructed in \cite{GYZZ}  to capture both the oscillation and decay properties only for the case with pure magnetic potential, that is, $a\equiv0$. It would be a bit delicate to recover the effect of the electric potential $a(\hat{x})|x|^{-2}$ in the explicitly construction.
\end{remark}

\begin{remark}
The perturbation argument in \cite{MZZ} to treat the inverse-square potential  breaks down since  the uniform Sobolev inequality $\frac1p-\frac1q=1$ in $\R^2$ fails, as in the free case. This is due to the presence of a logaritmic singularity in the kernel of the Green function, which also appear in the Aharonov-Bohm case (see \eqref{est:b} and \eqref{eq:log1} below).
\end{remark}

The proof of Theorem \ref{thm:LA0} follows from the ideas of Guti\'errez \cite{Gut} and Kenig-Ruiz-Sogge \cite{KRS}.
But we need to establish the Stein-Tomas restriction estimates associated with the variable coefficient operator $\LL_{{\A},0}$ (perturbated by
scaling critical magnetic potentials) which
has its own independent interesting. Since the operator $\LL_{{\A},0}$ is conically singular, we are in spirit of \cite{BFM18} (in which Blair-Ford-Marzuola studied the cluster estimates for a conical singular
Schr\"odinger operator) to prove the Stein-Tomas restriction estimates.
Based on the kernel of resolvent constructed in \cite{GYZZ},
due to the diffractive effect, the kernel is more singular and different from the kernel in the Euclidean setting.
In contrast to the estimates in weight $L^2$ space, the establishment of $L^p\to L^q$-type estimates needs to efficiently exploit the oscillation behavior of the kernel,
and hence we have to modify the celebrating oscillatory integral theory of Stein \cite{Stein} and H\"ormander \cite{Hor} to adapt to the singular kernel.\vspace{0.2cm}

%

The most usual applications of results like Theorem \ref{thm:LA0} are concerned with spectral theory. Uniform resolvent estimates are a fundamental tool in order to understand the point spectrum of Hamiltonians of the form
\begin{equation}\label{eq:complex}
\mathcal L_{\A,V}:=\LL_{{\A},0} + V,
\end{equation}
being $V:\R^2\to\C$ a (complex valued) measurable function. These non self-adjoint perturbations of self-adjoint Hamiltonians has been intensively studied in the last decades, due to their interest in Quantum Mechanics (see e.g. \cite{BB, SGY}). In our case, we consider the Aharonov-Bohm Schr\"odinger operator $\LL_{{\A},0}$ as the free Hamiltonian, and $V$ as the perturbation. A first natural question is concerned with the absence of eigenvalues for $\mathcal L_{\A,V}:=\LL_{{\A},0} + V$, when $V$ is sufficiently small in some suitable sense. Notice that in the magnetic-free case the question is critical in dimension $n=2$, due to the lack of Hardy's inequality. In other words, even a compactly supported potential $V$ can create eigenvalues in the spectrum of $-\Delta+V$. On the other hand, due to the presence of the Aharonov-Bohm field, a Hardy inequality holds when the circulation $\alpha$ is not an integer number (see \cite{LW}), hence the above question makes sense. An answer to this question is given in \cite{CFK, FKV2}, where weighted $L^2$ estimates come into play. In our play, we would like to use the same arguments as in Theorem 2 by \cite{Fr}, but unfortunately they strongly rely on the uniform Sobolev estimate, which in our case is false, as remarked above.

On the other hand, another natural problem is concerned with localization estimates for eigenvalues of $\mathcal L_{\A,V}$, in terms of integral norms of $V$. Following Frank (cfr. \cite[Theorem 1]{Fr}), we prove the following application of Theorem \ref{thm:LA0}.

\begin{theorem}\label{thm:frank}
Let $0<\gamma\leq\frac12$, and let $\lambda\in\C\setminus[0,\infty)$ be an eigenvalue of $\mathcal L_{\A,V}$ given in \eqref{eq:complex}. Then,
if the complex potential $V\in L^{\gamma+1}(\R^2)$, the following estimate
\begin{equation}\label{eq:eigenbound}
|\lambda|^{\gamma}\leq C_\gamma\int_{\R^2}|V(x)|^{\gamma+1}\,dx
\end{equation}
holds, with a positive constant $C_\gamma$ only depending on $\gamma$.
\end{theorem}
\begin{remark}\label{rem:frank1}
Inequality \eqref{eq:eigenbound} is related to Lieb-Thirring inequalities, which are usually simpler to prove when $V$ is real, since the self-adjoint setting provides useful variational characterizations for the eigenvalues (see e.g. \cite{BDEL} and the references therein). At our knowledge, the case $V\in\C$ was still not known.
Very recently, B\"ogli and Cuenin showed in \cite{BC} that the bound from above $\gamma\leq\frac12$ is sharp for the magnetic-free case, disproving a conjecture by Laptev and Safronov \cite{LS}. 
\end{remark}
\begin{remark}\label{rem:frank2}
In the magnetic-free case, Frank and Simon obtained in \cite{FS} that the bound can be improved to $0<\gamma\leq 2$, if one involves a stronger norm for $V$ in the radial direction. The argument, which is again perturbative, also works in magnetic operators as above. We omit here further details and we address the reader to \cite{FS}.
\end{remark}
\begin{remark}\label{rem:frank3}
By interpolation between \eqref{est:resa-z} and the trivial estimate
$$
\|(\LL_{{\A},0}-\sigma)^{-1}f\|_{L^2}\leq \text{dist}\,(\sigma,\R^+)\|f\|_{L^2},
$$
one can obtain bounds for the resolvent (and hence for the eigenvalues) for any value of $\gamma$, as showed in \cite{Fr2} in the magnetic-free case. Recently, Kwon and Lee showed in \cite{KL} that the resolvent bounds are sharp, while B\"oglie and Cuenin showed the sharpness of the corresponding eigenvalue bounds in \cite{BC}.
\end{remark}

The proof of Theorem \ref{thm:frank} is identical to the one of Theorem 1 in \cite{Fr}. Indeed, the only ingredients which come into play are the resolvent estimates \eqref{est:resa-z} and the Birman-Schwinger Principle, which in this case holds true (cfr. \cite[Theorems 6,7]{HK}). Therefore, we omit here further details of the proof.

The paper is organized as follows. In Section \ref{sec:pre}, as a preliminaries, we recall the resolvent and spectral measure kernel of the operator $\LL_{{\A},0}$  and some basic lemmas about the oscillatory integrals. In Section \ref{sec:thmmain}, we reduce the proof of Theorem \ref{thm:LA0} to the limitaton absorbing theorem, i.e, Theorem \ref{thm:LA0'} below, which will be shown in  Section \ref{sec:thmla0}. Finally, in appendix section, we verify some conditions used  in Section  \ref{sec:thmla0}. 
\vspace{0.2cm}

{\bf Acknowledgement.} The authors wish to express their gratitude to Rupert Frank for several useful remarks about Theorem \ref{thm:frank}.




\section{Preliminaries}\label{sec:pre}

In this section, we recall
the resolvent and spectral measure kernel of the operator $\LL_{{\A},0}$
constructed in  \cite{GYZZ} and some basic lemmas about the oscillatory integrals.
These preliminaries will be used in the next section.

\subsection{Some notations}
We recall some notations in \cite{FZZ, GYZZ}  which
are associated with the operator $\mathcal{L}_{{\A},0}$.
From \eqref{LAa} with $a=0$, we write
\begin{equation}\label{LAa-r}
\begin{split}
\mathcal{L}_{{\A},0}=-\partial_r^2-\frac{1}r\partial_r+\frac{L_{{\A},0}}{r^2},
\end{split}
\end{equation}
where the operator
\begin{equation}\label{L-angle}
\begin{split}
L_{{\A},0}&=(i\nabla_{\mathbb{S}^{1}}+{\A}(\hat{x}))^2,\qquad \hat{x}\in \mathbb{S}^1
\\&=-\Delta_{\mathbb{S}^{1}}+\big(|{\A}(\hat{x})|^2+i\,\mathrm{div}_{\mathbb{S}^{1}}{\A}(\hat{x})\big)+2i {\A}(\hat{x})\cdot\nabla_{\mathbb{S}^{1}}.
\end{split}
\end{equation}
Let $\hat{x}=(\cos\theta,\sin\theta)$, then
\begin{equation*}
\partial_\theta=-\hat{x}_2\partial_{\hat{x}_1}+\hat{x}_1\partial_{\hat{x}_2},\quad \partial_\theta^2=\Delta_{\mathbb{S}^{1}}.
\end{equation*}
Define $\alpha(\theta):[0,2\pi]\to \R$ such that
\begin{equation}\label{equ:alpha}
\alpha(\theta)={\bf A}(\cos\theta,\sin\theta)\cdot (-\sin\theta,\cos\theta),
\end{equation}
then by \eqref{eq:transversal}, we can write
\begin{equation*}
{\bf A}(\cos\theta,\sin\theta)=\alpha(\theta)(-\sin\theta,\cos\theta),\quad \theta\in[0,2\pi].
\end{equation*}
Thus, we obtain
\begin{equation}\label{LAa-s}
\begin{split}
L_{{\A},0}&=-\Delta_{\mathbb{S}^{1}}+\big(|{\A}(\hat{x})|^2+i\,\mathrm{div}_{\mathbb{S}^{1}}{\A}(\hat{x})\big)+2i {\A}(\hat{x})\cdot\nabla_{\mathbb{S}^{1}}\\
&=-\partial_\theta^2+\big(|{\alpha}(\theta)|^2+i\,{\alpha'}(\theta)\big)+2i {\alpha}(\theta)\partial_\theta\\
&=(i\partial_\theta+\alpha(\theta))^2.
\end{split}
\end{equation}
For simplicity, we define the constant $\alpha$ to be
$$\alpha=\Phi_{\A}=\frac1{2\pi}\int_0^{2\pi} \alpha(\theta) d\theta.$$

\subsection{Resolvent kernel}
To prove Theorem \ref{thm:LA0}, we need the representation of  the resolvent kernel of the operator $\LL_{{\A},0}$.
Let $\lambda>0$, we define the resolvent
\begin{equation}\label{def:res}
\big(\LL_{{\A},0}-(\lambda^2\pm i0)\big)^{-1}
=\lim_{\epsilon\searrow0}\big(\LL_{{\A},0}-(\lambda^2\pm i\epsilon)\big)^{-1}.
\end{equation}

\begin{proposition}[Resolvent kernel, \cite{GYZZ}]\label{prop:res-ker}
Let $x=r_1(\cos\theta_1,\sin\theta_1)$ and $y=r_2(\cos\theta_2,\sin\theta_2)$,
then we have the expression of resolvent kernel
\begin{align}\label{equ:res-ker-out}
     \big(\LL_{{\A},0}-(\lambda^2\pm i0)\big)^{-1}=&\frac1{\pi}\int_{\R^2}\frac{e^{-i(x-y)\cdot\xi}}{|\xi|^2-(\lambda^2\pm i0)}\;d\xi\, A_\alpha(\theta_1,\theta_2)\\\nonumber
  &+\frac1{\pi}\int_0^\infty \int_{\R^2}\frac{e^{-i{\bf n}\cdot\xi}}{|\xi|^2-(\lambda^2\pm i0)}\;d\xi \, B_\alpha(s,\theta_1,\theta_2)\;ds\\
  :=&G(r_1,\theta_1;r_2,\theta_2)+D(r_1,\theta_1;r_2,\theta_2),
\end{align}
where  ${\bf n}=(r_1+r_2, \sqrt{2r_1r_2(\cosh s-1)})$
and where
\begin{equation}\label{A-al}
A_{\alpha}(\theta_1,\theta_2)= \frac{e^{i\int_{\theta_1}^{\theta_2}\alpha(\theta')d\theta'}}{4\pi^2}\big(\mathbbm{1}_{[0,\pi]}(|\theta_1-\theta_2|)
  +e^{-i2\pi\alpha}\mathbbm{1}_{[\pi,2\pi]}(|\theta_1-\theta_2|)\big)
  \end{equation}
 and
\begin{equation}\label{B-al}
\begin{split}
&B_{\alpha}(s,\theta_1,\theta_2)= -\frac{1}{4\pi^2}e^{-i\alpha(\theta_1-\theta_2)+i\int_{\theta_2}^{\theta_{1}} \alpha(\theta') d\theta'}  \Big(\sin(|\alpha|\pi)e^{-|\alpha|s}\\
&\qquad +\sin(\alpha\pi)\frac{(e^{-s}-\cos(\theta_1-\theta_2+\pi))\sinh(\alpha s)-i\sin(\theta_1-\theta_2+\pi)\cosh(\alpha s)}{\cosh(s)-\cos(\theta_1-\theta_2+\pi)}\Big).
 \end{split}
\end{equation}

\end{proposition}

The above integral in $\xi$-variable in \eqref{equ:res-ker-out} is connected with the Euclidean resolvent  kernel which is known
as Hankel function, for example, see \cite[Chapter 3.4]{CK1} or \cite[Eq. (5.16.3)]{Le}.
\begin{lemma}\label{lem:res-k} The integral
\begin{equation}
\int_{\R^2} \frac{e^{ix\cdot\xi}}{|\xi|^2-(1\pm i0)} d\xi=\pm\frac{i}4H_0^{\pm}(|x|)
\end{equation}
where $H_0^\pm$ are the Hankel functions of order zero with $H_0^-=\overline{H_0^+}$ and
\begin{equation}
H_0^+(r)= e^{ir} r^{-1/2} a(r) +b(r).
\end{equation}
Here
\begin{equation}\label{est:a}
|a^{(k)}(r)|\lesssim r^{-k},\quad \forall k\geq 0, \quad a(r)=0, \quad \forall 0<r<\frac12,
\end{equation}
and $b(r)=0$ for all $r>\frac34$, with
\begin{equation}\label{est:b}
\begin{split}
|b(r)|\lesssim |\log r|,\quad |b^{(k)}(r)|&\lesssim r^{-k}\qquad \forall k\geq 1.
\end{split}
\end{equation}
\end{lemma}

\subsection{Spectral measure kernel}
According to Stone's formula, the spectral measure is related to the resolvent
\begin{equation}\label{equ:spemes}
   dE_{\sqrt{\LL_{\A},0}}(\lambda)=\frac{d}{d\lambda}E_{\sqrt{\LL_{\A},0}}(\lambda)\;d\lambda
   =\frac{\lambda}{i\pi}\big(R(\lambda+i0)-R(\lambda-i0)\big)\;d\lambda
\end{equation}
where the resolvent
$$R(\lambda\pm i0)=\big(\LL_{{\A},0}-(\lambda^2\pm i0)\big)^{-1}=\lim_{\epsilon\searrow0}\big(\LL_{{\A},0}-(\lambda^2\pm i\epsilon)\big)^{-1}.$$
\begin{proposition}[Spectral measure kernel \cite{GYZZ}]\label{prop:spect}
Let $x=r_1(\cos\theta_1, \sin\theta_1)$ and $y=r_2(\cos\theta_2, \sin\theta_2)$ in $\R^2\setminus\{0\}$. Define
\begin{equation}\label{d-j}
d(r_1,r_2,\theta_1,\theta_2)=\sqrt{r_1^2+r_2^2-2r_1r_2\cos(\theta_1-\theta_2)}=|x-y|,
\end{equation}
and
\begin{equation} \label{d-s}
d_s(r_1,r_2,\theta_1,\theta_2)=|{\bf n}|=\sqrt{r_1^2+r_2^2+2  r_1r_2\, \cosh s},\quad s\in [0,+\infty).
\end{equation}
Then the Schwartz kernel of the spectral measure
 \begin{equation}\label{ker:spect}
 \begin{split}
 dE_{\sqrt{\LL_{{\A},0}}}(\lambda;x,y) =&
\frac{\lambda}{\pi} \sum_{\pm}\Big(
 a_\pm(\lambda |x-y|)e^{\pm i\lambda |x-y|} A_{\alpha}(\theta_1,\theta_2)
 \\&\quad
+\int_0^\infty a_\pm(\lambda |{\bf{n}}|)e^{\pm i\lambda |{\bf n}|}
B_{\alpha}(s,\theta_1,\theta_2) ds\Big),
\end{split}
\end{equation}
where $A_{\alpha}(\theta_1,\theta_2)$ and $B_{\alpha}(s,\theta_1,\theta_2)$ are given in \eqref{A-al} and \eqref{B-al}
and $a_\pm\in C^\infty([0,+\infty))$ satisfies
\begin{equation}\label{bean}
\begin{split}
| \partial_r^k a_\pm(r)|\leq C_k(1+r)^{-\frac{1}2-k},\quad k\geq 0.
\end{split}
\end{equation}
\end{proposition}

\subsection{Some basic lemmas}
In this subsection, we recall two basic lemmas about the oscillatory integrals, see Stein \cite{Stein}.

\begin{lemma}\label{lem:absdec}
Let $\phi$ and $\psi$ be smooth functions so that $\psi$ has compact support in $(a,b)$, and $\phi'(x)\neq 0$ for all $x\in[a,b]$. Then,
\begin{equation}\label{equ:lamdec}
  \big|\int_a^b e^{i\lambda \phi(x)}\psi(x)\;dx\big|\lesssim (1+\lambda)^{-K},
\end{equation}
for all $K\geq0$.
\end{lemma}

We also need the following Van der Corput lemma.
\begin{lemma}[Van der Corput,\cite{Stein} ] \label{lem:VCL} Let $\phi$ be real-valued and smooth in $(a,b)$, and that $|\phi^{(k)}(x)|\geq1$ for all $x\in (a,b)$. Then
\begin{equation}
\left|\int_a^b e^{i\lambda\phi(x)}\psi(x)dx\right|\leq c_k\lambda^{-1/k}\left(|\psi(b)|+\int_a^b|\psi'(x)|dx\right)
\end{equation}
holds when (i) $k\geq2$ or (ii) $k=1$ and $\phi'(x)$ is monotonic. Here $c_k$ is a constant depending only on $k$.
\end{lemma}




\section{The proof of Theorem \ref{thm:LA0}}\label{sec:thmmain}

In this section, we prove Theorem \ref{thm:LA0} based on Proposition \ref{prop:res-ker}, Proposition \ref{prop:spect}  and Lemma \ref{lem:res-k}.
By the scaling invariant of the operator $\LL_{{\A},0}$, it suffices to prove \eqref{est:resa-z} when $|\sigma|=1$ with $\sigma\neq 1$, that is,
\begin{equation}\label{est:res-s}
\|(\LL_{{\A},0}-\sigma)^{-1}f\|_{L^q(\R^2)}\leq C\|f\|_{L^p(\R^2)}, \quad |\sigma|=1,\,\sigma\neq 1,
\end{equation}
where $p,q$ are in Theorem \ref{thm:LA0}.
More precisely, for fixed $\sigma\neq0$, from Proposition \ref{prop:spect}  with $\lambda\to |\sigma|^{1/2}\lambda$,
\begin{equation}
\begin{split}
(\LL_{{\A},0}-\sigma)^{-1}&=\int_0^\infty \frac{1}{\lambda^2-\sigma} dE_{\sqrt{\LL_{{\A},0}}}(\lambda;x,y)
\\&=
|\sigma|^{-1}\int_0^\infty \frac{1}{\lambda^2-\sigma |\sigma|^{-1}} dE_{\sqrt{\LL_{{\A},0}}}(\lambda;\sqrt{|\sigma|}x,\sqrt{|\sigma|} y),
\end{split}
\end{equation}
thus we see that \eqref{est:resa-z}
is reduce to \eqref{est:res-s}.

To prove \eqref{est:res-s}, let $\sigma=\pm \sqrt{1-\delta^2}+i\delta$ with $\delta\in (-1,0)\cup (0,1)$,
we consider three cases: (i) $\sigma=- \sqrt{1-\delta^2}+i\delta$, (ii) $\sigma=\sqrt{1-\delta^2}+i\delta$ and $|\delta|\geq \epsilon$,
and (iii) $\sigma=\sqrt{1-\delta^2}+i\delta$ and $|\delta|\leq \epsilon$, where $0<\epsilon\ll1$.

\subsection{Case (i): $\sigma=- \sqrt{1-\delta^2}+i\delta$.} Recall \eqref{ker:spect}, in this case, we have
\begin{equation}\label{R-case1}
\begin{split}
(\LL_{{\A},0}-\sigma)^{-1}&=\int_0^\infty \frac{\lambda}{\lambda^2+\sqrt{1-\delta^2}-i\delta} \sum_{\pm}\Big(
 a_\pm(\lambda |x-y|)e^{\pm i\lambda |x-y|} A_{\alpha}(\theta_1,\theta_2)
 \\&\quad
+\int_0^\infty a_\pm(\lambda |{\bf{n}}|)e^{\pm i\lambda |{\bf n}|}
B_{\alpha}(s,\theta_1,\theta_2) ds\Big) \, d\lambda.
\end{split}
\end{equation}

\begin{lemma}\label{lem:apmest} Let $a_\pm$ satisfy \eqref{bean}, then there exists a constant C such that
 \begin{equation}
 \begin{split}
\Big|\int_0^\infty \frac{\lambda}{\lambda^2+\sqrt{1-\delta^2}-i\delta}   a_\pm(\lambda r)e^{\pm i\lambda r}\, d\lambda\Big|\leq C\times
\begin{cases} -\log r, \quad r<3/4;\\
r^{-1},\quad r\geq 3/4.
\end{cases}
\end{split}
\end{equation}

\end{lemma}

\begin{proof}
We first consider the case that $r<3/4$. From \eqref{bean}, we obtain
 \begin{equation}
 \begin{split}
&\Big|\int_0^\infty \frac{\lambda}{\lambda^2+\sqrt{1-\delta^2}-i\delta}   a_\pm(\lambda r)e^{\pm i\lambda r}\, d\lambda\Big|\\
&\lesssim \int_0^1\lambda \,d\lambda+\int_1^{\frac1r} \lambda^{-1}\, d\lambda+r^{-\frac12}\int_{\frac1r}^\infty \lambda^{-\frac32}\, d\lambda\leq C|\log r|.
\end{split}
\end{equation}
Next we consider the case that $r\geq3/4$.  From \eqref{bean} again, and by integration by parts, we  obtain
 \begin{equation}
 \begin{split}
&\Big|\int_0^\infty \frac{\lambda}{\lambda^2+\sqrt{1-\delta^2}-i\delta}   a_\pm(\lambda r)e^{\pm i\lambda r}\, d\lambda\Big|
\lesssim r^{-1} \int_0^\infty \frac{1}{\lambda^2+1}\, d\lambda\leq Cr^{-1}.
\end{split}
\end{equation}
Thus, we conclude the proof of Lemma \ref{lem:apmest}.

\end{proof}

By direct computation, we have
\begin{equation}\label{est:AB}
|A_{\alpha}(\theta_1,\theta_2)|+\int_0^\infty \big|B_{\alpha}(s,\theta_1,\theta_2)\big| ds\leq C.
\end{equation}
This together with \eqref{R-case1} and the fact $|{\bf{n}}|\geq |x-y|$ gives
\begin{equation}\label{R-case1}
\begin{split}
|(\LL_{{\A},0}-\sigma)^{-1}|\lesssim |\log |x-y||\chi_{|x-y|\leq 3/4}+|x-y|^{-1} \chi_{|x-y|\geq 3/4}.
\end{split}
\end{equation}
On the other hand, it is easy to see that for any $1\leq p<+\infty$
\begin{align}\nonumber
  \big\|\log|x|\big\|_{L^p(|x|\leq \frac34)}^p= &C \int_0^\frac34 \big|\log r\big|^p r\;dr\\\label{equ:logxlp}
  =&C\int_{\frac43}^{+\infty}\big|\log s\big|^p\frac{ds}{s^3}<+\infty
\end{align}
which implies
$$\big|\log |x|\big|\chi_{|x|\leq 3/4}+|x|^{-1} \chi_{|x|\geq 3/4}\in L^{\big(1+\frac1q-\frac1p\big)^{-1}}(\R^2),$$
for $\frac23\leq \frac1p-\frac1q<1$ or $0<1+\frac1q-\frac1p\leq \frac13$. Hence, \eqref{est:res-s} follows by
Young's inequality.

\subsection{Case (ii): $\sigma=\sqrt{1-\delta^2}+i\delta$ and $|\delta|\geq \epsilon$.}
In this case, we use the same argument of Case (i) to prove \eqref{est:res-s} if we could show
\begin{lemma} Let $a_\pm$ satisfy \eqref{bean} and $1\geq |\delta|\geq \epsilon$, then
there exists a constant $C_\epsilon>0$ such that
 \begin{equation}
 \begin{split}
\Big|\int_0^\infty \frac{\lambda}{\lambda^2-\sqrt{1-\delta^2}-i\delta}   a_\pm(\lambda r)e^{\pm i\lambda r}\, d\lambda\Big|\leq C_\epsilon\times
\begin{cases} -\log r, \quad r<3/4;\\
r^{-1},\quad r\geq 3/4.
\end{cases}
\end{split}
\end{equation}

\end{lemma}

\begin{proof}
We first consider the case that $r<3/4$. From \eqref{bean} and $|\delta|\geq \epsilon$, we obtain
 \begin{equation}
 \begin{split}
&\Big|\int_0^\infty \frac{\lambda}{\lambda^2-\sqrt{1-\delta^2}-i\delta}   a_\pm(\lambda r)e^{\pm i\lambda r}\, d\lambda\Big|\\
&\lesssim_\epsilon  \int_0^1\lambda \,d\lambda+\int_1^{\frac1r} \lambda^{-1}\, d\lambda+r^{-\frac12}\int_{\frac1r}^\infty \lambda^{-\frac32}\, d\lambda\leq C_\epsilon |\log r|.
\end{split}
\end{equation}
Next we consider the case that $r\geq3/4$.  From \eqref{bean} again, and by Van der Corput lemma, we  obtain
 \begin{equation*}
\Big|\int_0^\infty \frac{\lambda}{\lambda^2-\sqrt{1-\delta^2}-i\delta}   a_\pm(\lambda r)e^{\pm i\lambda r}\, d\lambda\Big|
\lesssim_\epsilon r^{-1}\int_0^\infty \frac{1}{|\lambda^2-\sqrt{1-\delta^2}|+\delta}\, d\lambda\leq C_\epsilon r^{-1}.
\end{equation*}

\end{proof}

\subsection{Case (iii): $\sigma=\sqrt{1-\delta^2}+i\delta$ and $|\delta|\leq \epsilon$.}
In this case, we will prove
\begin{theorem}\label{thm:LA0'} Let $1\leq p<\frac43$ and $4<q\leq \infty$ satisfy $\frac23\leq \frac1p-\frac1q<1$ and let $\LL_{{\A},0}$
be \eqref{LAa} with $a\equiv0$.
Then, for $\lambda>0$, there exists a constant $C=C(p,q)>0$ such that
\begin{equation}\label{est:resa0}
\|(\LL_{{\A},0}-(\lambda\pm i0))^{-1}f\|_{L^q(\R^2)}\leq C\lambda^{\frac1p-\frac1q-1}\|f\|_{L^p(\R^2)},
\end{equation}
where $(\LL_{{\A},0}-(\lambda\pm i0))^{-1}:=\lim\limits_{\epsilon\to 0^+}(\LL_{{\A},0}-(\lambda\pm i\epsilon))^{-1}$.
In particular, if $6\leq q<\infty$, then \eqref{est:resa-z} holds for $q=p'$.
\end{theorem}
We postpone the proof of the key theorem in next section.
If taking $\epsilon$ small enough, we can
use Theorem \ref{thm:LA0'} with $\lambda=(1-\delta^2)^{\frac12}$ to obtain
\begin{equation}
\begin{split}
\|(\LL_{{\A},0}-\sigma)^{-1}f\|_{L^q(\R^2)}&\leq C(1-\delta^2)^{ (\frac1p-\frac1q-1)}\|f\|_{L^p(\R^2)}\leq C\|f\|_{L^p(\R^2)}.
\end{split}
\end{equation}

\section{The proof of Theorem \ref{thm:LA0'}}\label{sec:thmla0}
In this section, we prove the critical theorem about the limitaton absorbing theorem, that is, Theorem \ref{thm:LA0'}.
By the scaling invariant of the operator $\LL_{{\A},0}$ again, it suffices to prove \eqref{est:resa0} when $\lambda=1$, that is,
\begin{equation}\label{est:res1}
\|(\LL_{{\A},0}-(1\pm i0))^{-1}f\|_{L^q(\R^2)}\leq C\|f\|_{L^p(\R^2)}.
\end{equation}
From Proposition \ref{prop:res-ker} and Lemma \ref{lem:res-k}, we obtain
\begin{align}\label{eq:log1}
  (\LL_{{\A},0}-(1\pm i0))^{-1}(x,y)=&\frac{i}{4\pi}\big(G_1+G_2+D_1+D_2\big)(r_1,\theta_1;r_2,\theta_2),
\end{align}
where the kernels are defined by
\begin{equation}\label{ker-G}
\begin{split}
G_1(r_1,\theta_1;r_2,\theta_2)&=e^{i|x-y|} |x-y|^{-1/2} a(|x-y|)A_\alpha(\theta_1,\theta_2),\\
G_2(r_1,\theta_1;r_2,\theta_2)&=b(|x-y|)A_\alpha(\theta_1,\theta_2),
\end{split}
\end{equation}
and
\begin{equation}\label{ker-D}
\begin{split}
D_1(r_1,\theta_1;r_2,\theta_2)&= \int_0^\infty e^{i|{\bf n}|} |{\bf n}|^{-1/2} a(|{\bf n}|)\, B_\alpha(s,\theta_1,\theta_2)\;ds,\\
D_2(r_1,\theta_1;r_2,\theta_2)&= \int_0^\infty b(|{\bf n}|)\, B_\alpha(s,\theta_1,\theta_2)\;ds.
\end{split}
\end{equation}
Here $A_\alpha(\theta_1,\theta_2)$, $B_\alpha(s,\theta_1,\theta_2)$ are  as in Proposition \ref{prop:res-ker} and the functions $a$ and $b$
are as in Lemma \ref{lem:res-k}.
Hence, we write that
\begin{align*}
   & (\LL_{{\A},0}-(1\pm i0))^{-1}f(x)\\
=&\frac{i}{4\pi}\big(T_{G_1}f+T_{G_2}f+T_{D_1}f+T_{D_2}f\big)(x),
\end{align*}
where
$$T_{K}f(x)=\int_0^\infty \int_0^{2\pi }K(r_1,\theta_1;r_2,\theta_2) f(r_2,\theta_2)  d\theta_2 \,r_2 dr_2.$$
Thus, to prove \eqref{est:res1}, it suffices to prove that there exists a constant $C$ such that
\begin{equation}\label{equ:goalredu}
\|T_{K}\|_{L^p(\R^2)\to L^q(\R^2)}\leq C, \quad K\in\{G_1,G_2,D_1,D_2\},
\end{equation}
where $p,q$ are in Theorem \ref{thm:LA0'}.

\subsection{The estimations of $T_{G_2}$ and $T_{D_2}$.}

We first consider the two easy ones $T_{G_2}$ and $T_{D_2}$.
By \eqref{equ:logxlp}, \eqref{est:b} and the Young’s inequality, we deduce that for $1\leq q, p\leq \infty$ and $0\leq \frac1p-\frac1q<1$
\begin{equation}\label{equ:T2est}
\begin{split}
\|T_{G_2}f\|_{L^q(\R^2)}&\leq \Big\|\int_{\R^2} |b(|x-y|)| |f(y)|dy\Big\|_{L^q(\R^2)}\\
&\leq \Big\|\int_{\R^2} |\log|x-y||\chi_{|x-y|\leq \frac34} |f(y)|dy\Big\|_{L^q(\R^2)}\leq C\|f\|_{L^p(\R^2)}.
\end{split}
\end{equation}
And so it follows \eqref{equ:goalredu} with $K=G_2$.

Next, we turn to estimate the contribution from the term $T_{D_2}.$
We claim that
\begin{equation}\label{equ:g2claim}
  |D_2(r_1,\theta_1;r_2,\theta_2)|\lesssim |\log(r_1+r_2)|\chi_{r_1+r_2\leq \frac34}.
\end{equation}
Using this claim and the fact that $|x-y|\leq r_1+r_2$, we obtain \eqref{equ:goalredu} with $K=D_2$ by the same argument as \eqref{equ:T2est}.
It remains to prove the claim \eqref{equ:g2claim}.
Since $|{\bf n}|^2=r_1^2+r_2^2+2r_1r_2\cosh s\geq (r_1+r_2)^2$, from Proposition \ref{prop:res-ker} and Lemma \ref{lem:res-k}, one has
\begin{equation}
\begin{split}
&|D_2(r_1,\theta_1;r_2,\theta_2)|= \frac1{\pi}\Big|\int_0^\infty b(|{\bf n}|)\, B_\alpha(s,\theta_1,\theta_2)\;ds\Big|\\
&\leq C\int_0^\infty |\log(|{\bf n}|)|\chi_{|{\bf n}|\leq 1}\, \Big(|\sin(|\alpha|\pi)|e^{-|\alpha|s}\\\nonumber
  &\qquad +|\sin(\alpha\pi)|\cdot\Big|\frac{(e^{-s}-\cos(\theta_1-\theta_2+\pi))\sinh(\alpha s)-i\sin(\theta_1-\theta_2+\pi)\cosh(\alpha s)}{\cosh(s)-\cos(\theta_1-\theta_2+\pi)}\Big|\Big)\;ds.
\end{split}
\end{equation}
Note that $D_2(r_1,\theta_1,r_2,\theta_2)=0$ when $\alpha=0$. We always assume $\alpha\neq 0$.
Thus
\begin{equation}
\begin{split}
&|D_2(r_1,\theta_1;r_2,\theta_2)|\\
\leq& C |\log(r_1+r_2)|\chi_{r_1+r_2\leq \frac34} \int_0^\infty \Big(|e^{-|\alpha|s}\\\nonumber
  &\qquad +\Big|\frac{(e^{-s}-\cos(\theta_1-\theta_2+\pi))\sinh(\alpha s)-i\sin(\theta_1-\theta_2+\pi)\cosh(\alpha s)}{\cosh(s)-\cos(\theta_1-\theta_2+\pi)}\Big|\Big)\;ds\\
  \leq& C |\log(r_1+r_2)|\chi_{r_1+r_2\leq \frac34},
\end{split}
\end{equation}
where we have used the facts (proved in \cite{FZZ}) that for $\alpha\in(-1,1)\backslash\{0\}$
  \begin{align}\label{equ:ream1}
    \int_0^\infty e^{-|\alpha|s}\;ds\lesssim&1,\\\label{equ:ream2}
    \int_0^\infty  \Big|\frac{(e^{-s}-\cos(\theta_1-\theta_2+\pi))\sinh(\alpha s)}{\cosh(s)-\cos(\theta_1-\theta_2+\pi)}\Big|\;ds\lesssim&1,
    \\\label{equ:ream3}
     \int_0^\infty  \Big|\frac{\sin(\theta_1-\theta_2+\pi)\cosh(\alpha s)}{\cosh(s)-\cos(\theta_1-\theta_2+\pi)}\Big|\;ds\lesssim&1.
  \end{align}
Hence, we show the claim \eqref{equ:g2claim}. Therefore, we obtain \eqref{equ:goalredu} with $K=D_2$.

\subsection{The estimation of $ T_{G_1}$.} In this subsection, we prove
\begin{equation}\label{equ:tg1red}
  \|T_{G_1}\|_{L^{p}\to L^{q}}\leq C.
\end{equation}
 To consider the $T_{G_1}$  which is more complicated, we have to exploit the oscillation property.
To this end, recall $A_\alpha(\theta_1,\theta_2)$ in \eqref{A-al}, it suffices to prove
\begin{equation}
\begin{split}
\Big\|\int_0^\infty \int_0^{2\pi } K_G(r_1,r_2;\theta_1-\theta_2) f(r_2,\theta_2) r_2 dr_2 d\theta_2\Big\|_{L^q(\R^2)}\leq C\|f\|_{L^p(\R^2)}
\end{split}
\end{equation}
where
\begin{equation}
\begin{split}
K_G(r_1,r_2;\theta_1-\theta_2)&=e^{i|x-y|} |x-y|^{-1/2} a(|x-y|)
\mathbbm{1}_{I}(|\theta_1-\theta_2|),
\end{split}
\end{equation}
and $I=[0,\pi]$ or $[\pi,2\pi]$. Note that $|x-y|=\sqrt{r_1^2+r_2^2-2r_1r_2\cos(\theta_1-\theta_2)}$, we can only consider $I=[0,\pi]$ due to transmitting  invariant in $\theta$-variable.
Using the partition of one
$$\beta_0(r)=1-\sum_{j\geq1}\beta(2^{-j}r),\quad \beta\in\mathcal{C}_c^\infty\big(\big[\tfrac34,\tfrac83\big]\big),$$ we decompose
\begin{equation}\label{KGj}
\begin{split}
K_G(r_1,r_2;\theta_1-\theta_2)
=:\sum_{j\geq 0}K_G^j(r_1,r_2;\theta_1-\theta_2),
\end{split}
\end{equation}
where
$$K_G^j(r_1,r_2;\theta_1-\theta_2)=\beta(2^{-j}|x-y|)K_G(r_1,r_2;\theta_1-\theta_2), \quad j\geq1$$
and $K_G^0(r_1,r_2;\theta_1-\theta_2)=\beta_0(|x-y|)K_G(r_1,r_2;\theta_1-\theta_2)$.
Then we have
\begin{lemma}\label{lem:TGj} Let $T_{G_1}^j$ be the operator associated with kernel $K_{G}^j(r_1,r_2;\theta_1-\theta_2)$
given in \eqref{KGj} and let $6\leq 3p'\leq q$ and $6\leq q\leq \infty$.
Then
 \begin{equation}\label{est:TGj}
 \|T_{G_1}^j\|_{L^p\to L^q}\leq C 2^{-2j(\frac{1}{p}-\frac34)},\quad \forall j\geq0.
\end{equation}
\end{lemma}
We postpone the proof of \eqref{est:TGj} for a moment.
Now, we use Lemma \ref{lem:TGj} to prove \eqref{equ:tg1red}.
We divide the proof into several cases.

{\bf Case 1:} $\big(\tfrac1p,\tfrac1q\big)$ is in  the region $ABC$ of Figure 1. By Lemma \ref{lem:TGj}, we deduce that
for $1\leq p<\frac{4}{3}, q\geq 3p'$, there holds
\begin{equation}\label{equ:t1abc}
  \|T_{G_1}\|_{L^{p}\to L^{q}}\leq C\sum_{j\geq0} \|T_{G_1}^j\|_{L^{p}\to L^{q}}\leq C.
\end{equation}

{\bf Case 2:} $\big(\tfrac1p,\tfrac1q\big)$ is in  the region $A'B'C$ of Figure 1. By \eqref{equ:t1abc} and duality, we obtain
$$ \|T_{G_1}\|_{L^{p}\to L^{q}}\leq C$$ provided that $1\leq q'<\frac{4}{3}, p'\geq 3q$.

{\bf Case 3:} $\big(\tfrac1p,\tfrac1q\big)$ is in  the region $BB'C$ of Figure 1. By Case 1 and Case 2, we use the interpolation to get \eqref{equ:tg1red} for $\big(\tfrac1p,\tfrac1q\big)$ being in  the region $BB'C$ of Figure 1.

{\bf Case 4:} $\big(\tfrac1p,\tfrac1q\big)$ is in  the line $BB'$ of Figure 1, i.e.  $\frac1p-\frac1q=\frac23$. We follow the argument of \cite[Theorem 6]{Gut}. By the real interpolation, it suffices to prove
 \begin{equation}
 \|T_{G_1}\|_{L^{p,1}\to L^{q,\infty}}\leq C ,\quad (q,p)=\big(4,\tfrac{12}{11}\big)\, \text{and}\, \big(12,\tfrac43\big),
\end{equation}
where $L^{p,1}$ and $L^{q,\infty}$ are the Lorentz spaces. By \cite[Theorem V.3.13 and Theorem V.3.21]{Stein-W},
we need to prove
\begin{equation}\label{TGjB}
\sigma\big|\{x\in\R^2:|T_{G_1}\mathbbm{1}_E(x)|>\sigma\}\big|^{\frac1q}\leq C|E|^{\frac1p},\quad \forall \sigma>0, \,|E|<\infty,
\end{equation}
where $|E|$ is the measure of $E\subset \R^2$ and $\mathbbm{1}_E$ denotes its characteristic function as above. Let
$$A=\{x\in\R^2:|T_{G_1}\mathbbm{1}_E(x)|>\sigma\},$$
then we have
\begin{equation}\label{A}
\begin{split}
|A|&=\int_{A} dx\leq \frac1{\sigma}\int_{\R^2} |T_{G_1}\mathbbm{1}_E(x)|\mathbbm{1}_A(x) dx\\
&\leq \frac1{\sigma}\sum_{j\geq0}\int_{\R^2} |T^j_{G_1}\mathbbm{1}_E(x)|\mathbbm{1}_A(x) dx\\
&\leq \frac1{\sigma}\Big(\sum_{0\leq j\leq N}\|T^j_{G_1}\mathbbm{1}_E(x)\|_{L^2(\R^2)} |A|^{\frac12}+\sum_{ j\geq N+1}\|T^j_{G_1}\mathbbm{1}_E(x)\|_{L^\infty(\R^2)} |A|\Big),
\end{split}
\end{equation}
where $N\in\N$ will be determined lately.
From \eqref{est:TGj}, we have
 \begin{equation*}
 \|T_{G_1}^j\|_{L^1\to L^\infty}\leq C 2^{-\frac{j}2},\qquad  \|T_{G_1}^j\|_{L^2\to L^6}\leq C 2^{\frac{j}2}.
\end{equation*}
By duality, we also have
$$\|T_{G_1}^j\|_{L^{\frac65}\to L^2}\leq C 2^{\frac{j}2}.$$
We plug these into \eqref{A} to show
\begin{equation}
\begin{split}
\sigma |A|
&\leq C\Big( \sum_{0\leq j\leq N} 2^{\frac{j}2} |E|^{\frac56} |A|^{\frac12}+\sum_{ j\geq N+1}2^{-\frac j2}|E| |A|\Big)\\
&\leq C\Big( 2^{\frac{N}2} |E|^{\frac56} |A|^{\frac12}+2^{-\frac {N+1}2}|E| |A|\Big)\leq C|E|^{\frac{11}{12}} |A|^{\frac34}
\end{split}
\end{equation}
in which we choose $N$ such that $2^{\frac{N}2}\leq |E|^{\frac{1}{12}} |A|^{\frac14}\leq 2^{\frac{N+1}2}$. Therefore, we obtain
$$\sigma|A|^{\frac14}\leq C|E|^{\frac{11}{12}}$$ which implies \eqref{TGjB} with $(q,p)=\big(4,\tfrac{12}{11}\big)$.
To prove \eqref{TGjB} with $(q,p)=\big(12,\tfrac{4}{3}\big)$, we similarly show
\begin{align*}
   |A|
\leq& \frac1{\sigma}\Big(\sum_{0\leq j\leq N}\|T^j_{G_1}\mathbbm{1}_E(x)\|_{L^6(\R^2)} |A|^{\frac56}+\sum_{ j\geq N+1}\|T^j_{G_1}\mathbbm{1}_E(x)\|_{L^\infty(\R^2)} |A|\Big)\\
\lesssim& \frac1{\sigma}\Big( \sum_{0\leq j\leq N} 2^{\frac{j}2} |E|^{\frac12} |A|^{\frac56}+\sum_{ j\geq N+1}2^{-\frac j2}|E| |A|\Big)\\
\lesssim& \frac1{\sigma}\Big( 2^{\frac{N}2} |E|^{\frac12} |A|^{\frac56}+2^{-\frac {N+1}2}|E| |A|\Big)\leq C|E|^{\frac{3}{4}} |A|^{\frac{11}{12}}
\end{align*}
in which we choose $N$ such that $2^{\frac{N}2}\leq |E|^{\frac{1}{4}} |A|^{\frac1{12}}\leq 2^{\frac{N+1}2}$.
Hence, we obtain
$$\sigma|A|^{\frac1{12}}\leq C|E|^{\frac{3}{4}}$$ which gives \eqref{TGjB} with $(q,p)=\big(12,\tfrac{4}{3}\big)$.

In sum, we conclude the proof of the estimate \eqref{equ:tg1red} by using Lemma \ref{lem:TGj}. Therefore, it remains to show Lemma \ref{lem:TGj}.

\begin{proof}[{\bf The proof of Lemma \ref{lem:TGj}:}] The argument follows from the oscillatory integral theory by Stein \cite{Stein} and H\"ormander \cite{Hor}. However, there is a jump function $\mathbbm{1}_{[0,\pi]}(|\theta_1-\theta_2|)$ in the kernel so that the estimates are not trivial
consequence of the standard theory.

{\bf Step 1: $j=0$.}  Since $\beta_0\in L^{r}(\R^2)$ with $1\leq r<\infty$, for $1\leq q, p\leq \infty$ and $0\leq \frac1p-\frac1q<1$, then the Young inequality gives
$$ \|T_{G_1}^0f\|_{L^q}\leq \Big\|\int_{\R^2} \beta_0(|x-y|) f(y) dy\Big\|_{L^q(\R^2)}\leq C\|f\|_{L^p(\R^2)}.
$$

{\bf Step 2: $j\geq1.$}
Let $0<\epsilon\ll 1$ small enough, by finite covering lemma, we cover
$[0,2\pi]$ by $\Theta_\epsilon=\{\theta: |\theta-\theta_{k,0}|\leq \epsilon\}$ with fixed $\{\theta_{k,0}\}_{k=1}^{N}\in [0,2\pi]$.
 Noting that
$$K_G^j(r_1,r_2;\theta_1-\theta_2)=\beta(2^{-j}|x-y|)K_G(r_1,r_2;\theta_1-\theta_2),$$
and using translation invariance in norm $L^q_\theta$,  we only need to prove that  \eqref{est:TGj} holds for
\begin{equation}\label{assu:f}
\text{supp} f\subset \{(r_2,\theta_2): r_2\in(0,\infty),  \theta_2\in[0,\epsilon)\},\quad 0<\epsilon\ll 1.
\end{equation}
Note that if $\theta_2\in[0,\epsilon]$, then
\begin{equation}
\mathbbm{1}_{[0,\pi]}(|\theta_1-\theta_2|)=
\begin{cases}1\quad \text{when}\,\theta_1\in [0,\pi],\\
0\quad \text{when}\,\theta_1\in [\pi+\epsilon,2\pi].
\end{cases}
\end{equation}

We divide the proof into three cases.

{\bf Case 1: $\theta_1\in[\pi+\epsilon,2\pi].$}  Since
$$\mathbbm{1}_{[\pi+\epsilon,2\pi]}(\theta_1) K_G^j(r_1,r_2;\theta_1-\theta_2)=0,$$
we have $\mathbbm{1}_{[\pi+\epsilon,2\pi]}(\theta_1)T_{G_1}^j  f=0$. Thus,
$$\Big\|\mathbbm{1}_{[\pi+\epsilon,2\pi]}(\theta_1)T_{G_1}^j f\Big\|_{L^q(\R^2)}\leq C 2^{\frac{2j}{p'}}2^{-\frac j2}\|f\|_{L^p}.$$

{\bf Case 2: $\theta_1\in[0,\pi].$}
In this case, we have
 $$\mathbbm{1}_{[0,\pi]}(\theta_1) K_G^j(r_1,r_2;\theta_1-\theta_2)=\mathbbm{1}_{[0,\pi]}(\theta_1) \beta(2^{-j}|x-y|)e^{i|x-y|} |x-y|^{-1/2} a(|x-y|).$$
Then, the jump function disappears. Hence, we can use the oscillation integral results to obtain, for $q=3p'$ and $1\leq p<4$
\begin{equation}
\begin{split}
&\Big\|\mathbbm{1}_{[0,\pi]}(\theta_1)T_{G_1}^j f\Big\|_{L^q(\R^2)}\\
\leq& \Big\|\int_{\R^2}\beta(2^{-j}|x-y|)e^{i|x-y|} |x-y|^{-1/2} a(|x-y|)f(y) dy \Big\|_{L^q(\R^2)}\\
\leq& C 2^{\frac{2j}{p'}}2^{-\frac j2}\|f\|_{L^p},
\end{split}
\end{equation}
in which we use \cite[Lemma 5.4]{Sogge} in the last inequality. Similarly, by Young's inequality, we obtain
\begin{equation}\label{equ:3.24es}
\begin{split}
&\Big\|\mathbbm{1}_{[0,\pi]}(\theta_1)T_{G_1}^j f\Big\|_{L^\infty(\R^2)}\\
\leq& \Big\|\int_{\R^2}\beta(2^{-j}|x-y|)e^{i|x-y|} |x-y|^{-1/2} a(|x-y|)f(y) dy \Big\|_{L^\infty(\R^2)}\\
\leq& C 2^{\frac{2j}{p'}}2^{-\frac j2}\|f\|_{L^p}.
\end{split}
\end{equation}
Therefore, by interpolation,  we obtain
\begin{equation}
\begin{split}
\Big\|\mathbbm{1}_{[0,\pi]}(\theta_1)T_{G_1}^j f\Big\|_{L^q(\R^2)}
\leq C 2^{-2j(\frac1p-\frac34)}\|f\|_{L^p},
\end{split}
\end{equation}
provided $q\geq3p'$ and $1\leq p<4$.

{\bf Case 3: $\theta_1\in[\pi,\pi+\epsilon].$}
Note $f$ in \eqref{assu:f}, we are left to prove
\begin{equation}\label{equ:case3pi}
 \Big\|\mathbbm{1}_{[\pi,\pi+\epsilon]}(\theta_1) T_{G_1}^j \mathbbm{1}_{[0,\epsilon]} f\Big\|_{L^q(\R^2)}\lesssim 2^{j(\frac2{p'}-\frac12)} \| f\|_{L^p(\R^2)}
\end{equation}
for $6\leq 3p'\leq q\leq+\infty,$ i.e. $\big(\tfrac1p,\tfrac1q\big)$ is in the region $ABC$ of Figure 2.

\begin{center}
 \begin{tikzpicture}[scale=0.6]
 \draw[->] (0,0) -- (6,0) node[anchor=north] {$\frac1p$};
\draw[->] (0,0) -- (0,3)  node[anchor=east] {$\frac1q$};
\path (3,-1.5) node(caption){Fig 2. };  

\draw { (2.5,0) -- (5,0) -- (2.5, 0.9) -- (2.5,0) } [fill=gray!90]; 
 \draw (0,0) node[anchor=north] {O}
 (2.5,0) node[anchor=north] {$\frac12$}
 (5,0) node[anchor=north] {$1$}
 (2.2,0.3) node[anchor=north] {A}
 (5.3,0.3) node[anchor=north] {B};

\draw  (0, 0.9) node[anchor=east] {$\frac16$}
(2.5, 0.9) node[anchor=east] {C};

 \draw[thick]
 (2.5, 0.9) -- (5,0)
 (2.5, 0.9) -- (2.5,0)
(-0.05, 0.9)-- (0.05, 0.9)
(2.5,-0.05) -- (2.5,0.05)  (5,-0.05) -- (5,0.05);  


\end{tikzpicture}

\end{center}

If $\big(\tfrac1p,\tfrac1q\big)$ is in the line $AB$, i.e. $q=+\infty.$ By Young's inequality, we have
\begin{equation}
\begin{split}
&\Big\|\mathbbm{1}_{[\pi,\pi+\epsilon]}(\theta_1) T_{G_1}^j \mathbbm{1}_{[0,\epsilon]} f\Big\|_{L^\infty(\R^2)}\\
\lesssim&\Big\| \sup_{x\in\R^2}\int_{\R^2}\beta(2^{-j}|x-y|) |x-y|^{-1/2} |f(y)| dy\Big\|_{L^\infty(\R^2)} \leq C2^{\frac{2j}{p'}}2^{-\frac j2}\|f\|_{L^p(\R^2)}.
\end{split}
\end{equation}

By interpolation, we are reduced to show that \eqref{equ:case3pi} holds for $(p,q)=(2,6)$,
that is,
\begin{equation}\label{equ:case3pis2}
 \Big\|\mathbbm{1}_{[\pi,\pi+\epsilon]}(\theta_1) T_{G_1}^j \mathbbm{1}_{[0,\epsilon]} f\Big\|_{L^6(\R^2)}\lesssim 2^{\frac{j}2} \| f\|_{L^2(\R^2)}.
\end{equation}
We remark that one could utilize $TT^\ast$-argument by considering $L^2$-space.

Recall
$$K^j_G(r_1,r_2;\theta_1-\theta_2)=\beta(2^{-j}|x-y|)e^{i|x-y|} |x-y|^{-1/2} a(|x-y|)
\mathbbm{1}_{[0,\pi]}(|\theta_1-\theta_2|),$$
and let
$$\tilde{K}^j_G(r_1,r_2;\theta_1-\theta_2)=K^j_G(2^jr_1,2^jr_2;\theta_1-\theta_2).$$
Define the operator
\begin{equation}
\begin{split}
\tilde{T}_{G_1}^j f &=\int_0^\infty\int_0^{2\pi} \tilde{K}^j_G(r_1,r_2;\theta_1-\theta_2)\eta(\theta_1-\theta_2) f(r_2,\theta_2) \,d\theta_2 \,r_2 dr_2,
\end{split}
\end{equation}
where $ \eta(s) \in C_c^\infty ([\pi-2\epsilon,\pi+2\epsilon])$ such that $\eta(s)=1$ when $s\in[\pi-\epsilon,\pi+\epsilon]$.
One can check that
$$\mathbbm{1}_{[\pi,\pi+\epsilon]}T_{G_1}^j \big(\mathbbm{1}_{[0,\epsilon]} f(r_2,\theta_2)\big)(r_1,\theta_1)=2^{2j}\mathbbm{1}_{[\pi,\pi+\epsilon]}\tilde{T}_{G_1}^j \big(\mathbbm{1}_{[0,\epsilon]} f(2^jr_2,\theta_2)\big)(2^{-j}r_1,\theta_1).$$

To prove \eqref{equ:case3pis2}, thanks to the two characteristic cut off functions and the above scaling argument,  it suffices to prove
\begin{equation}\label{est:q62}
\big\|\tilde{T}_{G_1}^j f\big\|_{L^6(\R^2)} \leq C2^{-\frac{5j}6 }\|f\|_{L^{2}(\R^2)}.
\end{equation}
To this end, we write the operator
\begin{equation*}
\begin{split}
\tilde{T}_{G_1}^j f(r_1,\theta_1) &=\int_0^\infty\int_0^{2\pi} \tilde{K}^j(r_1,r_2;\theta_1-\theta_2)\eta(\theta_1-\theta_2) f(r_2,\theta_2) \,d\theta_2 \,r_2 dr_2\\
&\triangleq\int_0^\infty[\tilde{T}_{G_1,r_1,r_2}^j f(r_2, \cdot)]( \theta_1) \,r_2 dr_2
\end{split}
\end{equation*}
where the operator
\begin{equation}\label{Tr1r2}
[\tilde{T}_{G_1,r_1,r_2}^j g(\cdot)]( \theta_1)=\int_0^{2\pi} \tilde{K}^j(r_1,r_2;\theta_1-\theta_2) \eta(\theta_1-\theta_2) g(\theta_2) d\theta_2.
\end{equation}
Then we can prove \eqref{est:q62} by showing
\begin{lemma}\label{lem:key1} Let $\tilde{T}_{G_1,r_1,r_2}^j$ be the operator defined in \eqref{Tr1r2}. Then
\begin{equation}\label{est:inf1}
\begin{split}
&\Big\|\big[\big(\tilde{T}_{G_1,r_1,r_2}^j\big)^* \tilde{T}_{G_1,r_1,r'_2}^j \big( F(r'_2, \cdot)\big)\big]( \theta_2) \Big\|_{L^\infty (d\theta_2)}
\\\lesssim& 2^{-j}
(2^j|r_2-r_2'|)^{-1/2}r_1^{-1} \|F\|_{L^{1}( d\theta'_2)},
\end{split}
\end{equation}
and
\begin{equation}\label{est:22}
\begin{split}
&\Big\|\big[\big(\tilde{T}_{G_1,r_1,r_2}^j\big)^* \tilde{T}_{G_1,r_1,r'_2}^j \big( F(r'_2, \cdot)\big)\big]( \theta_2) \Big\|_{L^2 (d\theta_2)}
\\\lesssim& 2^{-j}
(2^jr_1r_2)^{-1/2}(2^jr_1r'_2)^{-1/2} \|F\|_{L^{2}( d\theta'_2)}.
\end{split}
\end{equation}

\end{lemma}
We assume Lemma \ref{lem:key1} for a while and prove \eqref{est:q62} now.  By interpolation with \eqref{est:inf1} and \eqref{est:22}, we obtain
\begin{equation}\label{equ:derivestan}
\begin{split}
&\Big\|\big[\big(\tilde{T}_{G_1,r_1,r_2}^j\big)^* \tilde{T}_{G_1,r_1,r'_2}^j \big( F(r'_2, \cdot)\big)\big]( \theta_2) \Big\|_{L^6(d\theta_2)}
\\\lesssim& 2^{-\frac{5j}3}
r_1^{-1}|r_2-r_2'|^{-\frac13}(r_2r_2')^{-\frac16} \|F\|_{L^{\frac65}(d\theta'_2)}.
\end{split}
\end{equation}

Since $\theta_1-\theta_2\in [\pi-3\epsilon,\pi+\epsilon]$, we get
$$|x-y|=\sqrt{r_1^2+r_2^2-2r_1r_2\cos(\theta_1-\theta_2)}\sim r_1+r_2.$$
This together with the support of $\beta$ in the kernel yields that
\begin{equation}\label{r-sim-1}
r_1+r_2\sim |x-y|\sim 1.
\end{equation}
Similarly, we have  $r_1+r_2'\sim 1$.
By \eqref{equ:derivestan}, the Hardy-Littlewood-Sobolev inequality and the fact that $r_2, r_2'\lesssim 1$, we prove
\begin{align}\nonumber
   &\Big\|\int_0^\infty\big[\big(\tilde{T}_{G_1,r_1,r_2}^j\big)^* \tilde{T}_{G_1,r_1,r'_2}^j \big( F(r'_2, \cdot)\big)\big]( \theta_2)  \,r'_2 dr'_2 \Big\|_{L^6(r_2dr_2 d\theta_2)}
\\\nonumber\lesssim& \Big\|\int_0^\infty\big\|\big[\big(\tilde{T}_{G_1,r_1,r_2}^j\big)^* \tilde{T}_{G_1,r_1,r'_2}^j \big( F(r'_2, \cdot)\big)\big]( \theta_2)\big\|_{L^6(d\theta_2)}  \,r'_2 dr'_2 \Big\|_{L^6(r_2dr_2)}\\\nonumber
\lesssim&2^{-\frac{5j}3}
r_1^{-1}\Big\|\int_0^\infty |r_2-r_2'|^{-\frac13}(r_2r_2')^{-\frac16} \|F\|_{L^{\frac65}(d\theta'_2)}\,r'_2 dr'_2 \Big\|_{L^6(r_2dr_2)}\\\label{est:65}
\lesssim& 2^{-\frac{5j}3}
r_1^{-1} \|F\|_{L^{\frac65}(r'_2dr'_2 d\theta'_2)}.
\end{align}

From \eqref{Tr1r2}, we note that
\begin{equation*}
\begin{split}
\Big(\big(\tilde{T}_{G_1}^j\big)^*\,\tilde{T}_{G_1}^j F\Big)(r_2,\theta_2)=\int_0^\infty\int_0^\infty\big[\big(\tilde{T}_{G_1,r_1,r_2}^j\big)^* \tilde{T}_{G_1,r_1,r'_2}^j \big( F(r'_2, \cdot)\big)\big]( \theta_2)  \,r'_2 dr'_2 \,r_1 dr_1,
\end{split}
\end{equation*}
therefore by \eqref{est:65} and the fact $r_1\lesssim 1$, we obtain
\begin{equation}
\begin{split}
&\Big\|\Big(\big(\tilde{T}_{G_1}^j\big)^*\,\tilde{T}_{G_1}^j F\Big)(r_2,\theta_2)\Big\|_{L^6(\R^2)} \lesssim 2^{-\frac{5j}3}\|F(r_2',\theta_2')\|_{L^{\frac{6}5}(\R^2)},
\end{split}
\end{equation}
which gives \eqref{est:q62} by $TT^\ast$ argument. We have completed the proof of \eqref{equ:case3pis2} once we prove Lemma \ref{lem:key1}.

\end{proof}

\begin{proof}[{\bf The proof of Lemma \ref{lem:key1}:}] The main idea is to remove the jump function $\mathbbm{1}_{[0,\pi]}(|\theta_1-\theta_2|)$.
From \eqref{Tr1r2}, the kernel of $\big(\tilde{T}_{G_1,r_1,r_2}^j\big)^* \tilde{T}_{G_1,r_1,r'_2}^j $ is
\begin{equation}
\begin{split}
 &\tilde{K}^j_G(r_1,r_2, r_2';\theta_2,\theta_2')\\
 =&\int_0^{2\pi} \tilde{K}^j_G(r_1,r_2;\theta_1-\theta_2) \overline{\tilde{K}^j_G(r_1,r'_2;\theta_1-\theta'_2)} \eta(\theta_1-\theta_2) \eta(\theta_1-\theta'_2) \,d\theta_1.
\end{split}
\end{equation}
Recall $x=r_1(\cos\theta_1,\sin\theta_1)$, $y=r_2(\cos\theta_2,\sin\theta_2)$ and let $z=r'_2(\cos\theta'_2,\sin\theta'_2)$,
then
\begin{align}\label{equ:kjkernj}
 &\tilde{K}^j_G(r_1,r_2, r_2';\theta_2,\theta_2')
\\\nonumber
=&2^{-j}\int_{0}^{2\pi} e^{i2^j(|x-y|-|x-z|)}\beta(|x-y|) |x-y|^{-1/2} \beta(|x-z|) |x-z|^{-1/2}\\\nonumber
&\qquad a(2^j|x-y|)a(2^j|x-z|)\mathbbm{1}_{[0,\pi]}(|\theta_1-\theta_2|)
\mathbbm{1}_{[0,\pi]}(|\theta_1-\theta_2'|)\eta(\theta_1-\theta_2)  \eta(\theta_1-\theta_2') d\theta_1.
\end{align}
Note $\theta_1-\theta_2\in [\pi-2\epsilon,\pi+2\epsilon]$, thus $\theta_1-\theta_2>0$. Thus $$\mathbbm{1}_{[0,\pi]}(|\theta_1-\theta_2|), \mathbbm{1}_{[0,\pi]}(|\theta_1-\theta_2'|)\implies \theta_1\leq \min\{\theta_2+\pi, \theta_2'+\pi\}.$$
Let $a=\pi-2\epsilon$ and $ b=\min\{\theta_2+\pi, \theta_2'+\pi\}$,
therefore we can drop off the characteristic jump function to obtain
\begin{equation}\label{Kj}
\begin{split}
 &\tilde{K}^j_G(r_1,r_2, r_2';\theta_2,\theta_2')
\\=&2^{-j}\int_{a}^{b} e^{i2^j(|x-y|-|x-z|)}\beta(|x-y|) |x-y|^{-1/2} \beta(|x-z|) |x-z|^{-1/2}\\
&\qquad \qquad a(2^j|x-y|)a(2^j|x-z|)\eta(\theta_1-\theta_2)  \eta(\theta_1-\theta_2') d\theta_1.
\end{split}
\end{equation}
{\bf Proof of \eqref{est:inf1}:} \eqref{est:inf1} easily follows from the following  claim
\begin{equation}\label{equ:tildkest1}
  |\tilde{K}^j_G(r_1,r_2, r_2';\theta_2,\theta_2')|\lesssim 2^{-j} (2^j|r_2-r_2'|)^{-\frac12}r_1^{-1}.
\end{equation}
To show the above claim, we denote
 $$d_G(r_1,r_2; \theta_1-\theta_2)=|x-y|=\sqrt{r_1^2+r_2^2-2r_1r_2\cos(\theta_1-\theta_2)},$$ and
\begin{equation}
\begin{split}
\Phi(r_1,r_2, r_2';\theta_1,\theta_2,\theta_2')
=d_G(r_1,r_2; \theta_1-\theta_2)-d_G(r_1,r'_2; \theta_1-\theta'_2).
\end{split}
\end{equation}
Due to the compact support of $\eta$,  we note that
$\theta_1-\theta_2\to \pi$ as $\epsilon\to 0$. Hence the facts $|x-y|, |x-z|\sim 1$ again implies $r_1+r_2, r_1+r_2'\sim 1$.
Then we compute that
\begin{equation}\label{1-d}
\begin{split}
\partial_{\theta_1}d_G(r_1,r_2; \theta_1-\theta_2)&=\frac{r_1r_2\sin(\theta_1-\theta_2)}{\sqrt{r_1^2+r_2^2-2r_1r_2\cos(\theta_1-\theta_2)}}\\
&=\frac{r_1r_2}{r_1+r_2}\sin(\theta_1-\theta_2)+O((r_1r_2)^2(\theta_1-\theta_2-\pi)^3),
\end{split}
\end{equation}
and
\begin{equation}\label{2-d}
\begin{split}
&\partial^2_{\theta_1}d_G(r_1,r_2; \theta_1-\theta_2)=\frac{r_1r_2}{r_1+r_2}\cos(\theta_1-\theta_2)+O((r_1r_2)^2(\theta_1-\theta_2-\pi)^2).\end{split}
\end{equation}
Therefore we obtain that
\begin{equation*}
\begin{split}
&\partial_{\theta_1}\Phi(r_1,r_2, r_2';\theta_1,\theta_2,\theta_2')=\frac{r_1r_2}{r_1+r_2}\sin(\theta_1-\theta_2)-\frac{r_1r_2'}{r_1+r_2'}\sin(\theta_1-\theta_2')+Error\\
&=\frac{r_1^2(r_2-r_2')}{(r_1+r_2)(r_1+r_2')}\sin(\theta_1-\theta_2)+\frac{r_1r'_2}{r_1+r'_2}\big(\sin(\theta_1-\theta_2)
-\sin(\theta_1-\theta'_2)\big)+Error,
\end{split}
\end{equation*}
and
\begin{equation*}
\begin{split}
&\partial^2_{\theta_1}\Phi(r_1,r_2, r_2';\theta_1,\theta_2,\theta_2')=\frac{r_1r_2}{r_1+r_2}\cos(\theta_1-\theta_2)-\frac{r_1r'_2}{r_1+r'_2}\cos(\theta_1-\theta'_2)+Error\\
&=\frac{r_1^2(r_2-r_2')}{(r_1+r_2)(r_1+r_2')}\cos(\theta_1-\theta_2)+\frac{r_1r'_2}{r_1+r'_2}\big(\cos(\theta_1-\theta_2)
-\cos(\theta_1-\theta'_2)\big)+Error.
\end{split}
\end{equation*}
where the error terms are $O(\epsilon^2)$ as $\epsilon\to 0$.

{\bf Case 1: $r_1r_2' |\theta_2-\theta_2'|\ll \epsilon^{-1}r_1^2|r_2-r_2'|$.} In this case, we have
\begin{equation*}
\Big|\frac{r_1r'_2}{r_1+r'_2}\big(\cos(\theta_1-\theta_2)-\cos(\theta_1-\theta'_2)\big)\Big|\lesssim r_1r_2'\epsilon |\theta_2-\theta_2'|\ll r_1^2|r_2-r_2'|.
\end{equation*}
Combining this inequality with the facts that $r_1+r_2, r_1+r_2'\sim 1$ and $\theta_1-\theta_2\to \pi$ as $\epsilon\to 0$, we obtain
 $$|\partial^2_{\theta_1}\Phi|\geq c |r_2-r_2'| r_1^2.$$
 Thus, the claim \eqref{equ:tildkest1} follows by the Van-der-Corput Lemma \ref{lem:VCL} in this case.

{\bf Case 2: $r_1r_2' |\theta_2-\theta_2'|\geq C \epsilon^{-1}r_1^2|r_2-r_2'|$.} In this case, one has
\begin{equation}
\begin{split}
\Big|\frac{r_1^2(r_2-r_2')}{(r_1+r_2)(r_1+r_2')}\sin(\theta_1-\theta_2)\Big|\lesssim& \epsilon r_1^2|r_2-r_2'|,\\
\Big|\frac{r_1r'_2}{r_1+r'_2}\big(\sin(\theta_1-\theta_2)-\sin(\theta_1-\theta'_2)\big)\Big|\geq& cr_1r_2'|\theta_2-\theta_2'|\geq \epsilon^{-1}r_1^2|r_2-r_2'|.
\end{split}
\end{equation}
Thus, we obtain $$|\partial_{\theta_1}\Phi|\geq c |r_2-r_2'| r_1^2$$
 when $\epsilon$ small enough.
Therefore, we prove  \eqref{equ:tildkest1} by Lemma \ref{lem:absdec}.

{\bf Proof of \eqref{est:22}:}  From $TT^\ast$-argument, we know that \eqref{est:22} is equivalent to
\begin{equation}\label{est:22'}
\Big\| \tilde{T}_{G_1,r_1,r_2}^j f \Big\|_{L^2 (d\theta_1)}\lesssim 2^{-\frac{j}2}
(2^jr_1r_2)^{-1/2} \|f\|_{L^{2}( d\theta_2)}.
\end{equation}
The assumption on the support of $\tilde{K}^j_G$ in $\theta_1-\theta_2$ means that we may treat the $\theta_j$ as variables in $\R$. Since $\tilde{T}_{G_1,r_1,r_2}^j$ is a convolution kernel in $\theta$, it suffices to prove that the corresponding Fourier multiplier satisfies
\begin{align}\nonumber
  &\Big|\int_{\pi-2\epsilon}^{\pi} e^{i\zeta\theta+i2^jd_G(r_1,r_2; \theta)}
  \beta\big(d_G(r_1,r_2; \theta)\big) \big(2^\frac{j}2d_G(r_1,r_2; \theta)\big)^{-\frac12} a\big(2^jd_G(r_1,r_2; \theta)\big)
\eta(\theta)\;d\theta\Big|\\\label{equ:kernconfrou}
\lesssim& 2^{-\frac{j}2}
(2^jr_1r_2)^{-1/2}.
\end{align}
Indeed, from \eqref{1-d} and \eqref{2-d}, we see
$$|\partial_{\theta_1}d_G(r_1,r_2; \theta_1-\theta_2)|+|\partial_{\theta_1}^2d_G(r_1,r_2; \theta_1-\theta_2)|\geq c r_1r_2.$$
Hence, we use Lemma \ref{lem:absdec} and the Van-der-Corput lemma \ref{lem:VCL} again to obtain \eqref{equ:kernconfrou}.
Thus, we prove Lemma \ref{lem:key1}.
\end{proof}

\subsection{The estimation of $T_{D_1}$.}  Now, we consider the operator $T_{D_1}$.
Recall
 \begin{equation}
\begin{split}
D_1(r_1,\theta_1;r_2,\theta_2)&= \frac1{\pi}\int_0^\infty e^{i|{\bf n}|} |{\bf n}|^{-1/2} a(|{\bf n}|)\, B_\alpha(s,\theta_1,\theta_2)\;ds
\end{split}
\end{equation}
and
\begin{equation*}
\begin{split}
&B_{\alpha}(s,\theta_1,\theta_2)= -\frac{1}{4\pi^2}e^{-i\alpha(\theta_1-\theta_2)+i\int_{\theta_2}^{\theta_{1}} \alpha(\theta') d\theta'}  \Big(\sin(|\alpha|\pi)e^{-|\alpha|s}\\
&\qquad +\sin(\alpha\pi)\frac{(e^{-s}-\cos(\theta_1-\theta_2+\pi))\sinh(\alpha s)-i\sin(\theta_1-\theta_2+\pi)\cosh(\alpha s)}{\cosh(s)-\cos(\theta_1-\theta_2+\pi)}\Big).
 \end{split}
\end{equation*}
Note that $D_2(r_1,\theta_1,r_2,\theta_2)=0$ when $\alpha=0$. We always assume $\alpha\neq 0$.

It suffices to prove the three estimates
\begin{equation}\label{equ:kdellred}
\Big\|\int_0^\infty \int_0^{2\pi } K^{\ell}_D(r_1,r_2;\theta_1-\theta_2) f(r_2,\theta_2) r_2 dr_2 d\theta_2\Big\|_{L^q(\R^2)}\leq C\|f\|_{L^p(\R^2)},
\end{equation}
for $\ell=1,2,3$, where
\begin{align}\label{kerD12}
K^1_D(r_1,r_2;\theta_1-\theta_2)&=\int_0^\infty e^{i|{\bf n}|} |{\bf n}|^{-1/2} a(|{\bf n}|)\, e^{-|\alpha|s}\;ds,\\\nonumber
K^2_D(r_1,r_2;\theta_1-\theta_2)&=\int_0^\infty e^{i|{\bf n}|} |{\bf n}|^{-1/2} a(|{\bf n}|)\, \frac{(e^{-s}-\cos(\theta_1-\theta_2+\pi))\sinh(\alpha s)}{\cosh(s)-\cos(\theta_1-\theta_2+\pi)}\;ds,
\end{align}
and
\begin{equation}\label{kerD3}
K^3_D(r_1,r_2;\theta_1-\theta_2)=\int_0^\infty e^{i|{\bf n}|} |{\bf n}|^{-1/2} a(|{\bf n}|)\, \frac{\sin(\theta_1-\theta_2+\pi)\cosh(\alpha s)}{\cosh(s)-\cos(\theta_1-\theta_2+\pi)}\;ds.
\end{equation}
Using the partition of one again
$$\beta_0(r)=1-\sum_{j\geq1}\beta(2^{-j}r),\quad \beta\in\mathcal{C}_c^\infty\big(\big[\tfrac34,\tfrac83\big]\big),$$ we decompose
\begin{equation}\label{KDlj}
K^{\ell}_D(r_1,r_2;\theta_1-\theta_2)
=:\sum_{j\geq 0}K^{\ell,j}_D(r_1,r_2;\theta_1-\theta_2),
\end{equation}
where for $j\geq1$ and $\ell=1,2,3$
$$K_D^{\ell,j}(r_1,r_2;\theta_1-\theta_2)=\beta(2^{-j}(r_1+r_2))K^\ell_D(r_1,r_2;\theta_1-\theta_2)$$
and $K_D^{\ell,0}(r_1,r_2;\theta_1-\theta_2)=\beta_0(r_1+r_2)K^\ell_D(r_1,r_2;\theta_1-\theta_2)$.
Then we have
\begin{lemma}\label{lem:TDj} For $\ell=1,2,3$, let $T_{D_1}^{\ell,j}$ be the operator associated with kernel $K_{D}^{\ell,j}(r_1,r_2;\theta_1-\theta_2)$ given in \eqref{KDlj} and let $6\leq 3p'\leq q\leq \infty$.
Then
 \begin{equation}\label{est:TDj}
 \|T_{D_1}^{\ell,j}\|_{L^p\to L^q}\leq C 2^{-2j(\frac{1}{p}-\frac34)},\quad \forall j\geq0.
\end{equation}
\end{lemma}

With Lemma \ref{lem:TDj} in hand, we obtain \eqref{equ:kdellred} by the same argument as the proof of \eqref{equ:tg1red} from Lemma \ref{lem:TGj}. Now we prove Lemma \ref{lem:TDj}.

{\bf Step 1: $j=0$.} From \eqref{equ:ream1}, \eqref{equ:ream2} and \eqref{equ:ream3}, similarly arguing as $D_2(r_1,r_2,\theta_1,\theta_2)$, we have
$$|K_D^{\ell,0}(r_1,r_2;\theta_1-\theta_2)|\leq \beta_0(r_1+r_2)\in L^r(\R^2),\quad 1\leq r\leq\infty.$$
Then we conclude $$ \|T_{D_1}^{\ell,0}\|_{L^p\to L^q}\leq C.$$
\vspace{0.1cm}

{\bf Step 2: $j\geq1$.}  Using the scaling, for $j\geq1$ and $\ell=1,2,3$, we define that
$$\tilde{K}_D^{\ell,j}(r_1,r_2;\theta_1-\theta_2)=K_D^{\ell,0}(2^jr_1,2^jr_2;\theta_1-\theta_2)=\beta(r_1+r_2)
K^\ell_D(2^jr_1,2^jr_2;\theta_1-\theta_2)$$
and
\begin{equation}
\tilde{T}_{D_1}^{\ell,j} f(r_1,\theta_1)=\int_0^\infty\int_0^{2\pi} \tilde{K}_D^{\ell,j}(r_1,r_2;\theta_1-\theta_2) f(r_2,\theta_2) \,r_2dr_2\,d\theta_2.
\end{equation}
Then, we have
$$T_{D_1}^{\ell,j} f(r_1,\theta_1)=2^{2j}\tilde{T}_{D_1}^{\ell,j} \big(f(2^{j}r_2,\theta_2)\big)(2^{j}r_1,\theta_1).$$
Therefore, \eqref{est:TDj} is equivalent to
 \begin{equation}\label{est:TDjequ}
 \|\tilde{T}_{D_1}^{\ell,j}\|_{L^p\to L^q}\leq C 2^{-j(\frac{1}{2}+\frac2q)}.
\end{equation}

Now, we show a basic proposition.
\begin{proposition}\label{lem:kerlplq}
Let $T_K$ be defined by
$$T_Kf(r_1,\theta_1):=\int_0^\infty \int_{0}^{2\pi} K(r_1,r_2,\theta_1,\theta_2)f(r_2,\theta_2)\;d\theta_2\;r_2\;dr_2,$$
and the kernel $K(r_1,r_2,\theta_1,\theta_2)$ satisfies
\begin{equation}\label{equ:kercond}
  |K(r_1,r_2,\theta_1,\theta_2)|\lesssim 2^{-\frac j2}\big(1+2^jr_1r_2\big)^{-\frac12}\beta(r_1+r_2).
\end{equation}
Then, there holds
\begin{equation}\label{equ:tkpqest}
  \big\|T_Kf\big\|_{L^q(\R^2)}\lesssim 2^{-j(\frac{1}{2}+\frac2q)}\|f\|_{L^p(\R^2)}
\end{equation}
for $q>4$ and $q>p'$.
\end{proposition}

\begin{proof}
  By \eqref{equ:kercond} and Minkowski's inequality, we obtain
  \begin{align*}
    \big\|T_Kf\big\|_{L^q(\R^2)}\lesssim & 2^{-\frac j2}\Big\|\int_0^\infty \int_{0}^{2\pi} \big(1+2^jr_1r_2\big)^{-\frac12}\beta(r_1+r_2) |f(r_2,\theta_2)|
    \;d\theta_2\;r_2\;dr_2\Big\|_{L^q(\R^2)}\\
    \lesssim& 2^{-\frac j2}\int_{|y|\leq 1}\big\|(1+2^j|x|\cdot|y|)^{-\frac12}\big\|_{L^q_x(\R^2)}|f(y)|\;dy\\
     \lesssim& 2^{-\frac j2}2^{-\frac2qj}\int_{|y|\leq 1}\big\|(1+|x|)^{-\frac12}\big\|_{L^q_x(\R^2)}|y|^{-\frac2q}|f(y)|\;dy\\
     \lesssim&2^{-j(\frac 2q+\frac12)}\|f\|_{L^p}\big\||y|^{-\frac2q}\big\|_{L^{p'}(|y|\leq 1)}
     \lesssim2^{-j(\frac 2q+\frac12)}\|f\|_{L^p},
  \end{align*}
  where we need the assumption $q>4$ and $q>p'$ to guarantee the convergence of $\big\|(1+|x|)^{-\frac12}\big\|_{L^q_x(\R^2)}$ and
  $\big\||y|^{-\frac2q}\big\|_{L^{p'}(|y|\leq 1)}.$

\end{proof}

Now, we show the kernels $\tilde{K}_D^{\ell,j}(r_1,r_2;\theta_1-\theta_2)$ satisfies \eqref{equ:kercond}. To do this, for fixed $r_1, r_2$, $\theta_1,\theta_2$ and $j\geq1$, we define
\begin{equation}
\phi(r_1,r_2; s)=|{\bf n}|=\sqrt{r_1^2+r_2^2+2r_1r_2\cosh s},
\end{equation}
and
\begin{equation}
\begin{split}
\psi_1(r_1,r_2,\theta_1,\theta_2; s)&=|{\bf n}|^{-1/2} a(2^j|{\bf n}|)\, e^{-|\alpha|s},\\
\psi_2(r_1,r_2,\theta_1,\theta_2; s)&=|{\bf n}|^{-1/2} a(2^j|{\bf n}|)\, \frac{(e^{-s}-\cos(\theta_1-\theta_2+\pi))\sinh(\alpha s)}{\cosh(s)-\cos(\theta_1-\theta_2+\pi)},\\
\psi_3(r_1,r_2,\theta_1,\theta_2; s)&=|{\bf n}|^{-1/2} a(2^j|{\bf n}|)\, \frac{\sin(\theta_1-\theta_2+\pi)\cosh(\alpha s)}{\cosh(s)-\cos(\theta_1-\theta_2+\pi)}.
\end{split}
\end{equation}
Then by the definition, we have
\begin{equation}\label{KDl}
\begin{split}
\tilde{K}_D^{\ell,j}(r_1,r_2;\theta_1-\theta_2)&=2^{-\frac j2}\beta(r_1+r_2) \int_0^\infty e^{i2^j\phi(s)} \psi_\ell(s)\;ds,\quad \ell=1,2,3.
\end{split}
\end{equation}

\begin{lemma}\label{lem:ker-est12} For  $j\geq1$, then there holds
\begin{equation}\label{kerD-est}
\begin{split}
|\tilde{K}_{D}^{1,j}(r_1,r_2;\theta_1-\theta_2)|+|\tilde{K}_{D}^{2,j}(r_1,r_2;\theta_1-\theta_2)|\lesssim 2^{-\frac j2}\big(1+2^jr_1r_2\big)^{-\frac12}.
\end{split}
\end{equation}
\end{lemma}

Using Lemma \ref{lem:ker-est12} and Proposition \ref{lem:kerlplq},  we obtain \eqref{est:TDjequ} with $\ell=1,2.$

\begin{proof}[{\bf The proof of Lemma \ref{lem:ker-est12}}]:
If $2^jr_1r_2\lesssim1$,  then \eqref{kerD-est} follows by
 using \eqref{equ:ream1}, \eqref{equ:ream2} and \eqref{equ:ream3}.
So from now on, we assume $2^jr_1r_2\geq 1$ in the proof.
We first compute that
\begin{equation}
\begin{split}
\partial_s \phi&=\frac{r_1r_2\sinh s}{(r_1^2+r_2^2+2r_1r_2\cosh s)^{1/2}}\\
\partial^2_s\phi&=\frac{r_1r_2\cosh s}{(r_1^2+r_2^2+2r_1r_2\cosh s)^{1/2}}-\frac{(r_1r_2\sinh s)^2}{(r_1^2+r_2^2+2r_1r_2\cosh s)^{3/2}},
\end{split}
\end{equation}
therefore we obtain, for $0\leq s\leq 1$
\begin{equation}
\partial_s \phi (0)=0,\quad |\partial^2_s \phi|\geq c \frac{r_1r_2}{r_1+r_2} \gtrsim r_1r_2,
\end{equation}
and for $s\geq 1$
\begin{equation}
\partial_s \phi\geq c \frac{r_1r_2}{r_1+r_2} \gtrsim r_1r_2.
\end{equation}
One can verify that $\partial_s\phi(s)$ is monotonic on the interval $[1,\infty)$ and the facts that
\begin{equation}\label{psi-b-1}
\int_0^\infty |\psi_1'(s)|ds+\int_1^\infty|\psi_2'(s)|ds\lesssim 1.
\end{equation} Indeed, if the derivative hits on $|{\bf n}|^{-1/2} a(2^j|{\bf n}|)$ which is bounded, we again use \eqref{equ:ream1}, \eqref{equ:ream2} and \eqref{equ:ream3} to obtain \eqref{psi-b-1}.
If the derivative hits on $e^{-|\alpha|s}$, it brings harmlessness.
By using Van der Corput Lemma \ref{lem:VCL}, thus we prove
\begin{equation}
\begin{split}
|\tilde{K}_{D}^{1,j}(r_1,r_2;\theta_1-\theta_2)|\lesssim 2^{-\frac j2}\big(2^jr_1r_2\big)^{-\frac12}.
\end{split}
\end{equation}
However, if the derivative hits on
\begin{equation*}
\begin{split}
 \frac{(e^{-s}-\cos(\theta_1-\theta_2+\pi))\sinh(\alpha s)}{\cosh(s)-\cos(\theta_1-\theta_2+\pi)},
\end{split}
\end{equation*}
it is harmless when $s\in[1,\infty)$ but it becomes more singular near at $s=0$
so that we can not verify
\begin{equation*}
\int_0^1|\psi_2'(s)|ds\lesssim 1.
\end{equation*}
Hence, to estimate $\tilde{K}_{D}^{2,j}$, we need more argument to bound
\begin{equation}\label{psi-2'}
\begin{split}
2^{-\frac j2}\beta(r_1+r_2) \int_0^1 e^{i2^j\phi(s)} \psi_2(s)\;ds.
\end{split}
\end{equation}
To this aim, when $s$ is close to $0$, we replace
\begin{equation*}
\begin{split}
 \frac{(e^{-s}-\cos(\theta_1-\theta_2+\pi))\sinh(\alpha s)}{\cosh(s)-\cos(\theta_1-\theta_2+\pi)}\sim \frac{(-s+b^2)(\alpha s)}{\frac{s^2}2+b^2}
\end{split}
\end{equation*}
where $b=\sqrt{2}\sin\big(\frac{\theta_1-\theta_2+\pi}2\big)$.
Then for $0\leq s\leq 1$,  uniformly in $b$, we have
\begin{equation}\label{d-v1}
\begin{split}
\Big|\partial_s^{k}\Big(\frac{(e^{-s}-\cos(\theta_1-\theta_2+\pi))\sinh(\alpha s)}{\cosh(s)-\cos(\theta_1-\theta_2+\pi)}-\frac{(-s+b^2)(\alpha s)}{\frac {s^2}2+b^2}\Big)\Big|\lesssim 1,\quad k=0,1
\end{split}
\end{equation}
which will be verified in Lemma \ref{lem:vercond} in the appendix section.

Therefore  the difference term is accepted by using Van der Corput Lemma \ref{lem:VCL} again. Hence, instead of \eqref{psi-2'}, we need to control
\begin{equation}\label{psi-2'1}
\begin{split}
2^{-\frac j2}\beta(r_1+r_2) \int_0^1 e^{i2^j\phi(s)}|{\bf n}|^{-1/2} a(2^j|{\bf n}|)\frac{(-s+b^2)(\alpha s)}{\frac{s^2}2+b^2}\;ds
\end{split}
\end{equation}
which is bounded by
\begin{equation}
\begin{split}
\lesssim 2^{-\frac j2}\beta(r_1+r_2)& \Big(\Big|\int_0^1 e^{i2^j\phi(s)} |{\bf n}|^{-1/2} a(2^j|{\bf n}|)\;ds\Big|\\
&+\Big|\int_0^1 e^{i2^j\phi(s)} |{\bf n}|^{-1/2} a(2^j|{\bf n}|)\, \frac{b^2}{\frac{s^2}2+b^2}\;ds\Big|\\
&+\Big|\int_0^1 e^{i2^j\phi(s)} |{\bf n}|^{-1/2} a(2^j|{\bf n}|)\, \frac{sb^2}{\frac{s^2}2+b^2}\;ds\Big|\Big).
\end{split}
\end{equation}
Now we can bound
\begin{equation}
\begin{split}
\beta(r_1+r_2)& \Big(\int_0^1 \Big|\partial_s\big(|{\bf n}|^{-1/2} a(2^j|{\bf n}|)\big)\Big|\;ds\\
&+\Big|\int_0^1 \partial_s\Big(|{\bf n}|^{-1/2} a(2^j|{\bf n}|)\, \frac{b^2}{\frac{s^2}2+b^2}\Big)\;ds\Big|\\
&+\Big|\int_0^1\partial_s\Big(|{\bf n}|^{-1/2} a(2^j|{\bf n}|)\, \frac{sb^2}{\frac{s^2}2+b^2}\Big)\;ds\Big|\Big)\lesssim 1.
\end{split}
\end{equation}
Then by using Van der Corput Lemma \ref{lem:VCL} again, we obtain
\begin{equation}
\begin{split}
\Big|2^{-\frac j2}\beta(r_1+r_2) \int_0^1 e^{i2^j\phi(s)}|{\bf n}|^{-1/2} a(2^j|{\bf n}|)\frac{(-s+b^2)(\alpha s)}{\frac{s^2}2+b^2}\;ds\Big|\lesssim 2^{-\frac j2} (2^jr_1r_2)^{-\frac12}.
\end{split}
\end{equation}
Therefore we prove Lemma \ref{lem:ker-est12}.
\end{proof}

Now we are left to consider $\tilde{K}_D^{3,j}$ which is more complicated.
To this end, we define
\begin{equation}\label{psi-m}
\begin{split}
\psi_{3,m}(r_1,r_2,\theta_1,\theta_2; s)&=(r_1+r_2)^{-1/2} a(2^j(r_1+r_2))\, \frac{\sin(\theta_1-\theta_2+\pi)}{\frac{s^2}2+2\sin^2\big(\tfrac{\theta_1-\theta_2+\pi}2\big)},
\end{split}
\end{equation}
and $\psi_{3,e}=\psi_{3}-\psi_{3,m}$.
We further define the kernels
\begin{equation}\label{KDe3}
\begin{split}
\tilde{K}_{D,m}^{3,j}(r_1,r_2;\theta_1-\theta_2)&=2^{-\frac j2}\beta(r_1+r_2) \int_0^\infty e^{i2^j\phi(s)} \psi_{3, m}(s)\;ds,\\
\tilde{K}_{D,e}^{3,j}(r_1,r_2;\theta_1-\theta_2)&=2^{-\frac j2}\beta(r_1+r_2) \int_0^\infty e^{i2^j\phi(s)} \psi_{3, e}(s)\;ds.
\end{split}
\end{equation}
Similarly as Lemma \ref{lem:ker-est12}, we have

\begin{lemma}\label{lem:ker-est3} Let $\tilde{K}_{D,e}^{3,j}$ be given in \eqref{KDe3} and let
\begin{equation}\label{def-H}
H(r_1,r_2;\theta_1-\theta_2)=2^{-\frac j2}\beta(r_1+r_2) (r_1+r_2)^{\frac12} \int_0^\infty e^{i2^jr_1r_2 s^2} \psi_{3, m}(s)\;ds,
\end{equation}
 then
\begin{equation}\label{kerDe3-est}
\begin{split}
|\tilde{K}_{D,e}^{3,j}(r_1,r_2;\theta_1-\theta_2)|\lesssim 2^{-\frac j2}\big(1+2^jr_1r_2\big)^{-\frac12},
\end{split}
\end{equation}
\begin{equation}\label{kerDm3-est}
\begin{split}
\big|e^{-i2^j(r_1+r_2)}\tilde{K}_{D,m}^{3,j}(r_1,r_2;\theta_1-\theta_2)-H(r_1,r_2;\theta_1-\theta_2)\big|\lesssim 2^{-\frac j2}\big(1+2^jr_1r_2\big)^{-\frac12},
\end{split}
\end{equation}
and
\begin{equation}\label{kerDH3-est}
\begin{split}
|H(r_1,r_2;\theta_1-\theta_2)|\lesssim 2^{-\frac j2}\big(1+2^jr_1r_2\big)^{-\frac12}.
\end{split}
\end{equation}

\end{lemma}

Using Lemma \ref{lem:ker-est12} and Proposition \ref{lem:kerlplq},  we obtain \eqref{est:TDjequ} with $\ell=3.$

\begin{proof}[{\bf The proof of Lemma \ref{lem:ker-est3}:}] We first prove \eqref{kerDe3-est}. Arguing similarly as Lemma \ref{lem:ker-est12},
it suffices to show
\begin{equation}
\begin{split}
\big|\partial_s \psi_{3, e}(r_1,r_2,\theta_1,\theta_2; s)\big|\lesssim 1,
\end{split}
\end{equation}
uniformly in $r_1,r_2,\theta_1,\theta_2$, when $r_1+r_2\sim 1$. For our purpose, we have
\begin{align}\nonumber
&\int_0^1\Big|\partial_s\Big[\Big(|{\bf n}|^{-1/2} a(2^j|{\bf n}|)\Big)\Big(\frac{\sin(\theta_1-\theta_2+\pi)\cosh(\alpha s)}{\cosh(s)-\cos(\theta_1-\theta_2+\pi)}
-\frac{\sin(\theta_1-\theta_2+\pi)}{\frac{s^2}2+2\sin^2\big(\tfrac{\theta_1-\theta_2+\pi}2\big)}\Big)\Big]\Big|\,ds\\\label{d-v2}
 \lesssim& 1,
\end{align}
and
\begin{align}\nonumber
&\int_0^1\Big|\partial_s\Big[\Big(|{\bf n}|^{-1/2} a(2^j|{\bf n}|)-(r_1+r_2)^{-\frac12}a(2^j(r_1+r_2))\Big)\frac{\sin(\theta_1-\theta_2+\pi)}{\frac{s^2}2+2\sin^2
\big(\tfrac{\theta_1-\theta_2+\pi}2\big)}\Big]\Big|\, ds\\\label{d-v3}
\lesssim& 1,
\end{align}
which will be verified in Lemma \ref{lem:vercond} in appendix.

We next prove \eqref{kerDm3-est}. We will use the the Morse Lemma to write the phase function
in term of quadratic formula via making variable change.
Let $$\bar{\varphi}(s)=\frac{\big(r_1^2+r_2^2+2r_1r_2\cosh s\big)^{\frac12}}{r_1r_2}-\frac{r_1+r_2}{r_1r_2},$$
then $\bar{\varphi}(0)=\bar{\varphi}'(0)=0$ and
$$\bar{\varphi}''(0)=\frac{1}{r_1+r_2}\neq 0.$$
Let $$g(s)=2\int_0^1(1-t)\bar{\varphi}''(ts)dt,$$
then we can write $$\bar{\varphi}(s)=\frac12g(s)s^2.$$
Note that $g(0)=\bar{\varphi}''(0)\neq 0$, we make the variable changing
$$\tilde{s}=|g(s)|^{\frac12}s.$$
Hence
\begin{equation}\label{d-ss}
\frac{d\tilde{s}}{ds}=(r_1+r_2)^{-\frac12}+O(r_1r_2s^2),\quad \frac{d{s}}{d\tilde{s}}=(r_1+r_2)^{\frac12}+O(r_1r_2\tilde{s}^2),
\end{equation}
and
$$\partial_{\tilde{s}}s\big|_{s=0}=|\bar{\varphi}''(0)|^{-\frac12}.$$
We write
$$\varphi(s)-(r_1+r_2)=r_1r_2\bar{\varphi}(s)=r_1r_2\tilde{s}^2,$$
by \eqref{d-ss}, hence
\begin{align}
&2^{-\frac j2}\beta(r_1+r_2) \int_0^\infty e^{i2^j(\phi(s)-(r_1+r_2))}\psi_{3, m}(s)\;ds\\\nonumber
&=2^{-\frac j2}\beta(r_1+r_2) \int_0^\infty e^{i2^jr_1r_2\tilde{s}^2}\big(\psi_{3, m}(\tilde{s})+O(r_1r_2\tilde{s}^2\big)\big((r_1+r_2)^{\frac12}+O(r_1r_2\tilde{s}^2)\big) d\tilde{s}.
\end{align}
By using Van der Corput Lemma \ref{lem:VCL} as before, as desired, the difference is bounded by
$$\big|e^{-i2^j(r_1+r_2)}\tilde{K}_{D,m}^{3,j}(r_1,r_2;\theta_1-\theta_2)-H(r_1,r_2;\theta_1-\theta_2)\big|\lesssim 2^{-\frac j2}\big(1+2^jr_1r_2\big)^{-\frac12}.$$

We finally prove \eqref{kerDH3-est}.
Recall
$b=\sqrt{2}\sin\big(\frac{\theta_1-\theta_2+\pi}2\big)$ and the definitions \eqref{def-H} and \eqref{psi-m}, by scaling, it suffices to show
\begin{equation}
\Big|\beta(r_1+r_2)  a(2^j(r_1+r_2)) \int_0^\infty e^{i2^{j+1}r_1r_2 s^2} \, \frac{b}{s^2+b^2}
\;ds\Big|\lesssim \big(1+2^jr_1r_2\big)^{-\frac12}.
\end{equation}
This is the same to term $H$ in \cite[(35)]{BFM18} so that it follows from the same argument.

\end{proof}

\section{Appendix}
In this appendix, we verify \eqref{d-v1}, \eqref{d-v2} and \eqref{d-v3}. In fact, we can prove
\begin{lemma}\label{lem:vercond}
Let $b=\sqrt 2\sin\theta,~\textbf{n}=(r_1+r_2,\sqrt{r_1r_2(\cosh s-1)}),$ and the nonnegative numbers $r_1,r_2$ such that $r_1+r_2\sim1.$ Then for $0< s\leq1$ and $|\alpha|\leq1,$ we have
\begin{align}\label{equ:paskcosht1123}
  \Big|\partial_s^k\Big(\frac{(e^{-s}-\cos(2\theta))\sinh(\alpha s)}{\cosh(s)-\cos(2\theta)}-\frac{(-s+b^2)(\alpha s)}{s^2/2+b^2}\Big)\Big|\lesssim 1,~k=0,1,\\\label{equ:paskcosht2123}
  \int_0^1\Big|\partial_s\Big[\Big(|\textbf{n}|^{-1/2}a(2^j|\textbf{n}|)\Big)\Big(\frac{\sin(2\theta)\cosh(\alpha s)}{\cosh(s)-\cos(2\theta)}-\frac{\sin(2\theta)}{s^2/2+b^2}\Big)\Big]\Big|ds\lesssim 1,
\end{align}
and
\begin{equation}\label{equ:paskcosht3123}
\int_0^1\Big|\partial_s\Big[\Big(|\textbf{n}|^{-1/2}a(2^j|\textbf{n}|)-(r_1+r_2)^{-\frac{1}{2}}a(2^j(r_1+r_2))\Big)
\Big(\frac{\sin(2\theta)}{s^2/2+b^2}\Big)\Big]\Big|ds\lesssim1,
\end{equation}
where $|a^{(k)}(r)|\lesssim r^{-k},~\forall\; k\geq0,r>0.$
\end{lemma}

The above lemma can be reduced to the following lemma with $b=\sqrt 2\sin\theta$.

\begin{lemma}\label{lem:redlemma}
Let $\textbf{n}=(r_1+r_2,\sqrt{r_1r_2(\cosh s-1)}),$ and the nonnegative numbers $r_1,r_2$ such that $r_1+r_2\sim1.$ Then for $0< s\leq1$, $|\alpha|\leq1$ and $0\leq b\leq 2,$ we have
\begin{align}\label{equ:paskcosht1}
  \Big|\partial_s^k\Big(\frac{(e^{-s}-1+b^2)\sinh(\alpha s)}{2\sinh^2(s/2)+b^2}-\frac{(-s+b^2)(\alpha s)}{s^2/2+b^2}\Big)\Big|\lesssim 1,~k=0,1,\\\label{equ:paskcosht2}
  \int_0^1\Big|\partial_s\Big[\Big(|\textbf{n}|^{-1/2}a(2^j|\textbf{n}|)\Big)\Big(\frac{b\cosh(\alpha s)}{2\sinh^2(s/2)+b^2}-\frac{b}{s^2/2+b^2}\Big)\Big]\Big|ds\lesssim 1,
\end{align}
and
\begin{equation}\label{equ:paskcosht3}
\int_0^1\Big|\partial_s\Big[\Big(|\textbf{n}|^{-1/2}a(2^j|\textbf{n}|)-(r_1+r_2)^{-\frac{1}{2}}a(2^j(r_1+r_2))\Big)
\Big(\frac{b}{s^2/2+b^2}\Big)\Big]\Big|ds\lesssim1,
\end{equation}
where $|a^{(k)}(r)|\lesssim r^{-k},~\forall\; k\geq0,r>0.$
\end{lemma}

\begin{proof}
{\bf Case 1: $b=0$.}
In this case,  \eqref{equ:paskcosht2} and \eqref{equ:paskcosht3} are obvious, and
\[\partial_s^k\Big(\frac{(e^{-s}-1+b^2)\sinh(\alpha s)}{2\sinh^2(s/2)+b^2}-\frac{(-s+b^2)(\alpha s)}{s^2/2+b^2}\Big)=\frac{(e^{-s}-1)\sinh(\alpha s)}{2\sinh^2(s/2)}+2\alpha.\]
Noting that
\begin{align}\label{equ:easyfact1}
  |e^{-s}-1|\leq 2|s|,~|\sinh(s)|\sim |s|,\quad \forall\;|s|\leq 1,
\end{align}
we obtain
\begin{equation}\label{equ:alphsin1e}
  \Big|\frac{(e^{-s}-1)\sinh(\alpha s)}{2\sinh^2(s/2)}+2\alpha\Big|\lesssim1.
\end{equation}

Next, we turn to estimate the term
\begin{equation}\label{equ:paresalpter1}
\Big|\partial_s\Big(\frac{(e^{-s}-1)\sinh(\alpha s)}{2\sinh^2(s/2)}+2\alpha\Big)\Big|=\Big|\partial_s\Big(\frac{(e^{-s}-1)\sinh(\alpha s)}{\cosh(s)-1}+2\alpha\Big)\Big|.
\end{equation}
By a simple computation, we have
\begin{align*}
&\partial_s\Big(\frac{(e^{-s}-1)\sinh(\alpha s)}{\cosh(s)-1}+2\alpha\Big)\\
=&\frac{[-e^{-s}\sinh(\alpha s)+\alpha(e^{-s}-1)\cosh(\alpha s)](\cosh(s)-1)-(e^{-s}-1)\sinh(\alpha s)\sinh(s)}{(\cosh(s)-1)^2}.
\end{align*}
Using the Taylor expansion, we can write the above numerator as
\[
\begin{aligned}
&\Big[-\big(1-s+\frac{s^2}{2!}+o(s^2)\big)\big(\alpha s+\frac{(\alpha s)^3}{3!}+o(s^3)\big)\\
&\qquad+\alpha\big(-s+\frac{s^2}{2!}-\frac{s^3}{3!}+o(s^3)\big)\big(1+\frac{s^2}{2!}+o(s^3)\big)\Big]\big(\frac{s^2}{2!}+\frac{s^4}{4!}+o(s^3)\big)\\
&\qquad-\big(-s+\frac{s^2}{2!}-\frac{s^3}{3!}+o(s^3)\big)\big(\alpha s+\frac{(\alpha s)^3}{3!}+o(s^4)\big)\big(s+\frac{s^3}{3!}+o(s^4)\big)\\
=&\big(-\alpha s^3+\frac{3\alpha}{4}s^4+o(s^4)\big)-\big(-\alpha s^3+\frac{\alpha}{2}s^4+o(s^4)\big)\\
=&\frac{\alpha}{4}s^4+o(s^4),
\end{aligned}
\]
which shows that the term \eqref{equ:paresalpter1} is bounded. Thus,  we finish the proof of the case $b=0$.

{\bf Case 2: $b^2\geq\delta_0>0$.}
By a simple computation, for $k=0,1,$ we have
\begin{align*}
  \Big|\partial_s^k\Big(\frac{(e^{-s}-1+b^2)\sinh(\alpha s)}{2\sinh^2(s/2)+b^2}-\frac{(-s+b^2)(\alpha s)}{s^2/2+b^2}\Big)\Big|\lesssim& \frac{1}{\delta_0},\\
  \Big|\partial_s^k\Big(\frac{b\cosh(\alpha s)}{2\sinh^2(s/2)+b^2}-\frac{b}{s^2/2+b^2}\Big)\Big|\lesssim&\frac{1}{\delta_0},\\
  \Big|\partial_s^k\Big(\frac{b}{s^2/2+b^2}\Big)\Big|\lesssim&\frac{1}{\delta_0}.
\end{align*}
On the other hand, writing $r=2^j|\textbf{n}|,$  one has
\begin{equation}\label{equ:ajbfnest}
\begin{aligned}
&\partial_s(|\textbf{n}|^{-1/2}a(2^j|\textbf{n}|))\\
=&[-\frac{1}{2}a(2^j|\textbf{n}|)|\textbf{n}|^{-5/2}+|\textbf{n}|^{-3/2}(2^j\partial_r(2^j|\textbf{n}|))]\cdot r_1r_2\sinh(s).
\end{aligned}
\end{equation}
By the assumption $r_1+r_2\sim1,~r_1,r_2\geq0$ and the properties of $a(r),$ it is obvious that $|\partial_s(|\textbf{n}|^{-1/2}a(2^j|\textbf{n}|))|$ is bounded. Thus, one can obtain the case $b^2\geq\delta_0$ of the lemma.

{\bf Case 3: $0<b^2<\delta_0\ll1.$}
 Whereas the case $s\geq\delta_0$ is similar to $b^2\geq \delta_0,$ we can obtain the same conclusions as Case 2. Therefore, we may assume that $s,b^2<\delta_0.$

We first prove \eqref{equ:paskcosht1}.
 For $k=0,$ by \eqref{equ:easyfact1}, we get
 $$\Big|\frac{(e^{-s}-1+b^2)\sinh(\alpha s)}{2\sinh^2(s/2)+b^2}-\frac{(-s+b^2)(\alpha s)}{s^2/2+b^2}\Big|\lesssim \frac{(s+b^2)s}{s^2+b^2}\lesssim1.$$
 %

For $k=1,$ we have
\[
\begin{aligned}
&\partial_s\Big(\frac{(e^{-s}-1+b^2)\sinh(\alpha s)}{2\sinh^2(s/2)+b^2}-\frac{(-s+b^2)(\alpha s)}{s^2/2+b^2}\Big)\\
=&\partial_s\Big(\frac{(e^{-s}-1+b^2)\sinh(\alpha s)}{2\sinh^2(s/2)+b^2}-\frac{(-s+b^2)(\alpha s)}{s^2/2+b^2}\Big)\\
=&\frac{-e^{-s}\sinh(\alpha s)+\alpha(e^{-s}-1+b^2)\cosh(\alpha s)}{2\sinh^2(s/2)+b^2}-\sinh(s)\cdot\frac{(e^{-s}-1+b^2)\sinh(\alpha s)}{(2\sinh^2(s/2)+b^2)^2}\\
&-\frac{\alpha(b^2-2s)}{s^2/2+b^2}+\frac{\alpha(-s+b^2)s^2}{(s^2/2+b^2)^2}\\
\triangleq&A+B,
\end{aligned}
\]
where
\begin{align*}
 A=&\frac{-e^{-s}\sinh(\alpha s)+\alpha(e^{-s}-1+b^2)\cosh(\alpha s)}{2\sinh^2(s/2)+b^2}-\frac{\alpha(b^2-2s)}{s^2/2+b^2},\\
 B=&\sinh(s)\cdot\frac{(e^{-s}-1+b^2)\sinh(\alpha s)}{(2\sinh^2(s/2)+b^2)^2}+\frac{\alpha(-s+b^2)s^2}{(s^2/2+b^2)^2}.
\end{align*}
We will prove that $A$ and $B$ are bounded respectively by using the Taylor expansion as follows.

Write $\displaystyle A=\frac{A_1}{A_2}$ with
\begin{align*}
A_1&=\big[-e^{-s}\sinh(\alpha s)+\alpha(e^{-s}-1+b^2)\cosh(\alpha s)\big](s^2/2+b^2)-\alpha(-s+b^2)s^2\big(2\sinh^2(s/2)+b^2\big)\\
   &=\Big[-\big(1-s+\frac{s^2}{2}-\frac{s^3}{3!}+\frac{s^4}{4!}+o(s^4)\big)\big((\alpha s)+\frac{(\alpha s)^3}{3!}+o(s^4)\big)\\
&\qquad+\alpha\big(-s+\frac{s^2}{2}-\frac{s^3}{3!}+\frac{s^4}{4!}+o(s^4)+b^2\big)\big(1+\frac{(\alpha s)^2}{2}+\frac{(\alpha s)^4}{4!}+o(s^5)\big)\Big]\\
&\qquad~~\times\big(\frac{s^2}{2}+b^2\big)-\alpha(-s+b^2)s^2\big[\frac{s^2}{2}+\frac{s^4}{4!}+o(s^5)+b^2\big],
\end{align*}
\begin{align*}
A_2&=(2\sinh^2(s/2)+b^2)(s^2/2+b^2)\\
 &=\big[\frac{s^2}{2}+\frac{s^4}{4!}+o(s^5)+b^2\big]\big[\frac{s^2}{2}+b^2\big]\\
 &=\frac{s^4}{4}+s^2b^2+b^4+o(s^5).
\end{align*}
Write  $\displaystyle B=\frac{B_1}{B_2}$ with
\begin{align*}
B_1=&-\sinh(s)(e^{-s}-1+b^2)\sinh(\alpha s)(s^2/2+b^2)^2+\alpha(-s+b^2)s^2(2\sinh^2(s/2)+b^2)^2\\
=&-\Big(s+\frac{s^3}{3!}+o(s^4)\Big)\Big(-s+\frac{s^2}{2}-\frac{s^3}{3!}+\frac{s^4}{4!}+o(s^4)+b^2\Big)\Big((\alpha s)+\frac{(\alpha s)^3}{3!}+o(s^4)\Big)\Big(\frac{s^2}{2}+b^2\Big)^2\\
&+\alpha(-s+b^2)s^2\Big(\frac{s^2}{2}+\frac{s^4}{4!}+o(s^5)+b^2\Big)^2,\\
B_2=&(2\sinh^2(s/2)+b^2)^2(s^2/2+b^2)^2\\
=&\Big(\frac{s^2}{2}+\frac{s^4}{4!}+o(s^5)+b^2\Big)^2\cdot\Big(\frac{s^2}{2}
+b^2\Big)^2.
\end{align*}
Since $A_1,A_2$ and $B_1,B_2$  both have the same order about $s$ or $b$ respectively, we can deduce that $A$ and $B$ are bounded.

\vskip 0.1in
Now, we turn to prove \eqref{equ:paskcosht2}. Noting that
\begin{align}\nonumber
&\int_0^1\Big|\partial_s\Big[\Big(|\textbf{n}|^{-1/2}a(2^j|\textbf{n}|)\Big)\Big(\frac{b\cosh(\alpha s)}{2\sinh^2(s/2)+b^2}-\frac{b}{s^2/2+b^2}\Big)\Big]\Big|ds\\\nonumber
=&\int_0^1\Big|\partial_s\Big(|\textbf{n}|^{-1/2}a(2^j|\textbf{n}|)\Big)\cdot\Big(\frac{b\cosh(\alpha s)}{2\sinh^2(s/2)+b^2}-\frac{b}{s^2/2+b^2}\Big)\\\label{equ:paskcterm21}
&\quad +\Big(|\textbf{n}|^{-1/2}a(2^j|\textbf{n}|)\Big)\cdot\partial_s\Big(\frac{b\cosh(\alpha s)}{2\sinh^2(s/2)+b^2}-\frac{b}{s^2/2+b^2}\Big)\Big|ds
\end{align}
and using \eqref{equ:ajbfnest} and the fact that $|\textbf{n}|^{-1/2}a(2^j|\textbf{n}|)$ is bounded, and
\begin{equation}\label{equ:simfactint1}
\Big|\int_0^1\frac{b}{s^2+b^2}ds\Big|\lesssim1,
\end{equation} we only need to show
\begin{equation}\label{equ:cas31sk01}
\Big|\partial_s^k\Big(\frac{\cosh(\alpha s)}{2\sinh^2(s/2)+b^2}-\frac{1}{\frac{s^2}{2}+b^2}\Big)\Big|\lesssim \frac{1}{s^2+b^2},\quad k=0,1.
\end{equation}
It is easy to show the case $k=0$.

For $k=1,$ by a direct computation, we get
\begin{align}\nonumber
&\partial_s\Big(\frac{\cosh(\alpha s)}{2\sinh^2(s/2)+b^2}-\frac{1}{\frac{s^2}{2}+b^2}\Big)\\\nonumber
=&\frac{\alpha\sinh(\alpha s)}{2\sinh^2(s/2)+b^2}-\frac{2\cosh(\alpha s)\sinh(s/2)\cosh(s/2)}{[2\sinh^2(s/2)+b^2]^2}+\frac{s}{(\frac{s^2}{2}+b^2)^2}\\\label{equ:parfrscsinter1}
\triangleq& I_1-I_2+I_3.
\end{align}
It is easy to check that $$|I_1|\lesssim  \frac{1}{s^2+b^2}.$$
On the other hand,
\begin{align*}
   & |I_3-I_2| \\
  \leq & \Big|\frac{s}{(\frac{s^2}{2}+b^2)^2}-\frac{s}{[2\sinh^2(\frac{s}2)+b^2]^2}\Big|+\Big|\frac{s}{[2\sinh^2(\frac{s}2)+b^2]^2}-\frac{2\cosh(\alpha s)\sinh(\frac{s}2)\cosh(\frac{s}2)}{[2\sinh^2(\frac{s}2)+b^2]^2}\Big| \\
  \leq & s\frac{2[2\sinh^2(\frac{s}2)+2b^2+\frac{s^2}{2}]\cdot[\sinh(\frac{s}2)+\frac{s}2]\cdot|\sinh(\frac{s}2)-\frac{s}2|}
  {(\frac{s^2}{2}+b^2)^2[2\sinh^2(\frac{s}2)+b^2]^2}\\
  &+\frac{|s-2\cosh(\alpha s)\sinh(\frac{s}2)\cosh(\frac{s}2)|}{[2\sinh^2(\frac{s}2)+b^2]^2}\\
  \lesssim& \frac{1}{s^2+b^2}.
\end{align*}
Hence, we derive \eqref{equ:cas31sk01} with $k=1$.

\vskip 0.1in

Finally, we show \eqref{equ:paskcosht3}.
 By \eqref{equ:ajbfnest} and \eqref{equ:simfactint1}, we only need to consider the term
\begin{equation}\label{equ:remscase332term}
\int_0^1\Big|\Big(|\textbf{n}|^{-1/2}a(2^j|\textbf{n}|)-(r_1+r_2)^{-\frac{1}{2}}a(2^j(r_1+r_2))\Big)
\partial_s\Big(\frac{b}{s^2/2+b^2}\Big)\Big|ds.
\end{equation}
Since
\begin{align*}
&(r_1+r_2)^{-\frac{1}{2}}-|\textbf{n}|^{-1/2}\\
=&\frac{|\textbf{n}|^2-(r_1+r_2)^2}{(r_1+r_2)^{\frac{1}{2}}|\textbf{n}|^{1/2}((r_1+r_2)^{\frac{1}{2}}+|\textbf{n}|^{1/2})(r_1+r_2+|\textbf{n}|)}\\
=&\frac{2\sinh^2(\frac{s}{2})r_1r_2}{(r_1+r_2)^{\frac{1}{2}}|\textbf{n}|^{1/2}((r_1+r_2)^{\frac{1}{2}}+|\textbf{n}|^{1/2})(r_1+r_2+|\textbf{n}|)}\\
\lesssim&s^2,
\end{align*}
we get
\begin{align*}
   &\Big||\textbf{n}|^{-1/2}a(2^j|\textbf{n}|)-(r_1+r_2)^{-\frac{1}{2}}a(2^j(r_1+r_2))\Big| \\
  \leq & \Big|(r_1+r_2)^{-\frac{1}{2}}-|\textbf{n}|^{-1/2}\Big||a(2^j|\textbf{n}|)|+(r_1+r_2)^{-\frac{1}{2}}\big|a(2^j|\textbf{n}|)-
   a(2^j(r_1+r_2))\big|\\
  \lesssim & s^2+ \big||\textbf{n}|-(r_1+r_2)\big|\lesssim s^2.
\end{align*}
Plugging this into \eqref{equ:remscase332term} implies
\begin{align*}
  &\int_0^1\Big|\Big(|\textbf{n}|^{-1/2}a(2^j|\textbf{n}|)-(r_1+r_2)^{-\frac{1}{2}}a(2^j(r_1+r_2))\Big)
\partial_s\Big(\frac{b}{s^2/2+b^2}\Big)\Big|ds\\
\lesssim& \int_0^1 s^2 \frac{bs}{(s^2/2+b^2)^2}\;ds\lesssim \int_0^1 \frac{b}{s^2/2+b^2}\;ds\lesssim1.
\end{align*}
And so  \eqref{equ:paskcosht3} holds. This concludes the proof of the lemma.

\end{proof}

\begin{center}

\end{center}

\end{document}